\documentclass[a4paper,reqno,11pt, oneside]{amsart}

\usepackage{amssymb}
\usepackage{amstext}
\usepackage{amsmath}
\usepackage{amscd}
\usepackage{amsthm}
\usepackage{amsfonts}
\usepackage{enumerate}
\usepackage{graphicx}
\usepackage{color}
\usepackage{here} % 図や表を強制的に出力させる
\usepackage{bm} % 太字ベクトル

\usepackage{tikz}
\usepackage{mathtools}
\usepackage{latexsym}
\usepackage{booktabs}
\makeatletter
  
  \@addtoreset{equation}{section}
\makeatother

\newcommand{\calA}{\mathcal{A}}

\newcommand{\calC}{\mathcal{C}}
\newcommand{\calD}{\mathcal{D}}

\newcommand{\calH}{\mathcal{H}}

\newcommand{\calL}{\mathcal{L}}

\newcommand{\calO}{\mathcal{O}}

\newcommand{\calM}{\mathcal{M}}

\newcommand{\ZZ}{\mathbb{Z}}
\newcommand{\QQ}{\mathbb{Q}}
\newcommand{\RR}{\mathbb{R}}

\newcommand{\kk}{\Bbbk}

\newcommand{\ab}{\mathbf{a}}
\newcommand{\bb}{\mathbf{b}}
\newcommand{\cb}{\mathbf{c}}
\newcommand{\eb}{\mathbf{e}}

\newcommand{\tb}{\mathbf{t}}

\newcommand{\xb}{\mathbf{x}}
\newcommand{\yb}{\mathbf{y}}

\newcommand{\fkp}{\mathfrak{p}}

\newcommand{\Hom}{\operatorname{Hom}}

\def\opn#1#2{\def#1{\operatorname{#2}}} % to make operators
\opn\Cl{Cl} \opn\conv{conv} \opn\deg{deg} \opn\rank{rank} \opn\Spec{Spec} 
\opn\cone{cone} \opn\End{End} \opn\Hom{Hom} \opn\mod{mod} \opn\gldim{gldim} \opn\pdim{pdim}

\newtheorem{thm}{Theorem}[section]

\newtheorem{lem}[thm]{Lemma}
\newtheorem{prop}[thm]{Proposition}

\theoremstyle{definition}
\newtheorem{defi}[thm]{Definition}
\newtheorem{ex}[thm]{Example}

\theoremstyle{remark}
\newtheorem{rem}[thm]{Remark}

\begin{document}

\title{Conic divisorial ideals and non-commutative crepant resolutions of edge rings of complete multipartite graphs}
\author{Akihiro Higashitani}
\author{Koji Matsushita}

\address[A. Higashitani]{Department of Pure and Applied Mathematics, Graduate School of Information Science and Technology, Osaka University, Suita, Osaka 565-0871, Japan}
\email{higashitani@ist.osaka-u.ac.jp}
\address[K. Matsushita]{Department of Pure and Applied Mathematics, Graduate School of Information Science and Technology, Osaka University, Suita, Osaka 565-0871, Japan}
\email{k-matsushita@ist.osaka-u.ac.jp}

\subjclass[2010]{
Primary 13C14; %Cohen-Macaulay modules
Secondary 13F65, %Commutative rings defined by binomial ideals, toric rings, etc
14M25, %Toric varieties, Newton polyhedra
05C25, %Graphs and abstract algebra
52B20. %52B20 Lattice polytopes in convex geometry
} 
\keywords{class groups, conic divisorial ideals, non-commutative resolutions, non-commutative crepant resolutions, edge rings, complete multipartite graphs}

\maketitle

%%%---abstract---%%%
\begin{abstract} 
The first goal of the present paper is to study the class groups of the edge rings of complete multipartite graphs, denoted by $\kk[K_{r_1,\ldots,r_n}]$, where $1 \leq r_1 \leq \cdots \leq r_n$. 
More concretely, we prove that the class group of $\kk[K_{r_1,\ldots,r_n}]$ is isomorphic to $\ZZ^n$ if $n =3$ with $r_1 \geq 2$ or $n \geq 4$, 
while it turns out that the excluded cases can be deduced into Hibi rings. 
The second goal is to investigate the special class of divisorial ideals of $\kk[K_{r_1,\ldots,r_n}]$, called conic divisorial ideals. 
We describe conic divisorial ideals for certain $K_{r_1,\ldots,r_n}$ including all cases where $\kk[K_{r_1,\ldots,r_n}]$ is Gorenstein. 
Finally, we give a non-commutative crepant resolution (NCCR) of $\kk[K_{r_1,\ldots,r_n}]$ in the case where it is Gorenstein. 
\end{abstract}

\bigskip

\section{Introduction}

\subsection{Backgrounds}
The present paper has two goals: one is the study of conic divisorial ideals of certain toric rings, which are edge rings of complete multipartite graphs, 
and the other is the construction of their non-commutative crepant resolutions as the application of the study of conic divisorial ideals. 

Conic divisorial ideals are a certain class of divisorial ideals which are a special kind of maximal Cohen--Macaulay (MCM, for short) modules of rank one 
and play important roles in the theory of commutative rings with positive characteristic as well as non-commutative algebraic geometry. 
In fact, the following theorems hold. Let $R=\kk[C \cap \ZZ^d]$ be a normal affine monoid algebra, where $C \subset \RR^d$ is a pointed finitely generated normal cone: 
\begin{thm}[{\cite[Proposition 3.8]{BG1}, \cite[Proposition 3.2.3]{SmVdB}}]\label{thm1}
The set of all conic divisorial ideals of $R$ corresponds to the set of the $R$-modules appearing in $R^{1/k}$ as direct summands for $k \gg 0$, 
where $R^{1/k}=\kk[C \cap (1/k\ZZ)^d]$ is regarded as an $R$-module. 
\end{thm}
\begin{thm}[{\cite[Corollary 6.2]{FMS}, \cite[Proposition 1.8]{SpVdB}}]\label{thm2}
For $k \gg 0$, $\End_R(R^{1/k})$ is a non-commutative resolution (NCR) of $R$. 
\end{thm}
These theorems imply that the endomorphism ring of the direct sum of all conic divisorial ideals of $R$ is an NCR of $R$. 

Recently, conic divisorial ideals of certain toric rings and their applications to the construction of non-commutative (crepant) resolutions 
are well studied (see, e.g., \cite{HN, N, SpVdB}, and so on).  
One of the most important aspects of the study of conic divisorial ideals is that, as mentioned above, 
we can construct an NCR, which is a non-commutative ring having a finite global dimension. 
Such rings often appear in the context of representation theory (see, e.g., Auslander's work \cite{Aus}). 
Furthermore, conic divisorial ideals are used to analyze the structure of Frobenius push-forward of $R$ in the theory of commutative algebra with positive characteristic. 
Thus, it is quite natural to classify conic divisorial ideals of certain class of toric rings since there are many applications of such classifications. 
In the present paper, we focus on the toric rings arising from complete multipartite graphs. 

Non-commutative crepant resolution (NCCR) was introduced by Van den Bergh (\cite{VdB}) in the context of non-commutative algebraic geometry. 
As this name implies, an NCCR is strongly related to the usual crepant resolution used in algebraic geometry, 
and the introduction of NCCRs provides a new interaction among algebraic geometry, (non-)commutative ring theory, and representation theory. 
Since the existence of an NCCR does not hold in general, the existence of an NCCR for some certain classes is one of the most well-studied problems in this area. 
For example, concerning the case of toric rings, the following results are known: 
\begin{itemize}
\item an NCCR of each quotient singularity by a finite abelian group (which is a toric ring associated with a simplicial cone) is given (see, e.g., \cite{IW, VdB}); 
\item toric rings whose class groups are $\ZZ$ have an NCCR (\cite{VdB}); 
\item Hibi rings whose class groups are $\ZZ^2$ have an NCCR (\cite{N}); 
\item NCCRs of $3$-dimensional Gorenstein toric ring can be obtained via the theory of dimer models (see \cite{Bro, IU, SpVdB2}); 
\item an NCCR of Segre products of polynomial rings are constructed (\cite{HN}); 
\item there are other results on the existence of an NCCR for toric rings (see \cite{SpVdB, SpVdB1}). 
\end{itemize}
In the present paper, we study the existence of an NCCR for $\kk[K_{r_1,\ldots,r_n}]$. 

\smallskip

In the remaining parts of the present paper, let $\kk$ be an algebraically closed field of characteristic $0$, for simplicity.

\subsection{Edge rings and edge polytopes}
Throughout the present paper, all graphs are finite and have no loop and no multiple edge. 
Consider a graph $G$ on the vertex set $V(G)=[d]$, where $[d]:=\{1,\ldots,d\}$ for a positive integer $d$, with the edge set $E(G)=\{e_1,\ldots,e_r\}$. 
%Let $\kk[{\bf x}]=\kk[x_1,\ldots,x_r]$ (resp. $\kk[{\bf t}]=\kk[t_1,\ldots,t_d]$) be the polynomial ring in $r$ (resp. $d$) variables over a field $\kk$. 
Let $\kk[{\bf t}]=\kk[t_1,\ldots,t_d]$ be the polynomial ring in $d$ variables over a field $\kk$. 
We write $\kk[G]$ for the subalgebra of $\kk[{\bf t}]$ generated by ${\bf t}^e=t_it_j$ for all edges $e=\{i,j\} \in E(G)$. 
The monoid $\kk$-algebra $\kk[G]$ is called the \textit{edge ring} of $G$. 

%Let $G$ be a graph on the vertex set $V(G)=[d]$. 
Given an edge $e=\{i,j\} \in E(G)$, let $\rho(e)=\eb_i+\eb_j$, 
where $\eb_i$ denotes the $i$-th unit vector of $\RR^d$ for $i=1,\ldots,d$. 
We define the convex polytope associated to $G$ as follows: 
$$P_G=\conv(\{\rho(e) : e \in E(G)\}) \subset \RR^d.$$ 
We call $P_G$ the \textit{edge polytope} of $G$. 
Note that the edge ring is the toric ring (also known as the polytopal monomial subring) of the edge polytope. 
See, e.g., \cite[Section 10]{Villa} or \cite[Section 5]{HHO} for the introduction to edge rings.

\subsection{In the case of complete multipartite graphs}

A \textit{complete multipartite graph} is a graph on the vertex set $\bigsqcup_{i=1}^n V_i$ with the edge set $\{\{a,b\} : a \in V_i, b \in V_j, 1 \leq i < j \leq n\}$. 
When $|V_i|=r_i$ for $i=1,\ldots,n$, we denote it by $K_{r_1,\ldots,r_n}$. 
In the case $n=2$, we call $K_{r_1,r_2}$ a \textit{complete bipartite graph}. 
In the case $r_1=\cdots=r_n=1$, we call $K_{\underbrace{1,\ldots,1}_n}$, denoted by $K_n$, a \textit{complete graph}. 

The edge rings $\kk[K_{r_1,\ldots,r_n}]$ of complete multipartite graphs are investigated by Ohsugi--Hibi in \cite{OH00}. 
It is proved that the $\kk[K_{r_1,\ldots,r_n}]$ %(which is the defining ideal of the toric ring) 
is a normal Cohen--Macaulay domain and Koszul, and its Hilbert series is also explicitly computed. 
The algebras of {\em Segre--Veronese type} are introduced in \cite[Section 1]{OH00}, 
which are a simultaneous generalization of both Segre products and Veronese subrings of polynomial rings. 
It is proved in \cite[Proposition 1.2]{OH00} that the edge ring $\kk[G]$ of a graph $G$ is of Segre--Veronese type 
if and only if $G$ is a complete multipartite graph. 

Our main object of the present paper is the edge rings $\kk[K_{r_1,\ldots,r_n}]$ of complete multipartite graphs.

\subsection{Main Results}
The first main theorem of the present paper is the following: 
%\begin{thm}\label{main1}
%Let $1 \leq r_1 \leq \cdots \leq r_n$. 
%The class group of the edge ring of the complete multipartite graph $K_{r_1,\ldots,r_n}$ is as follows: 
%\begin{align*}
%\Cl(\kk[K_{r_1,\ldots,r_n}]) \cong \begin{cases}
%0     &\text{ if }n=2, r_1=1 \text{ or }n=3, r_1=r_2=r_3=1, \\
%\ZZ   &\text{ if }n=2, r_1 \geq 2, \\
%\ZZ^2 &\text{ if }n=3, r_1=1, r_2 \geq 2, \\
%\ZZ^n &\text{ if }n=3, r_1 \geq 2 \text{ or }n \geq 4. \\
%\end{cases}
%\end{align*}
%\end{thm}
\begin{thm}\label{main1}
Let $1 \leq r_1 \leq \cdots \leq r_n$. Assume that $n=3$ with $r_1 \geq 2$ or $n \geq 4$. 
Then the class group $\Cl(\kk[K_{r_1,\ldots,r_n}])$ of $\kk[K_{r_1,\ldots,r_n}]$ is isomorphic to $\ZZ^n$ as groups. 
\end{thm}
\begin{rem}\label{rem}
In Theorem~\ref{main1}, the cases where $n=2$ and $n=3$ with $r_1=1$ are not discussed. 
However, those cases can be deduced to Hibi rings (see Proposition~\ref{prop:order}). 
Here, Proposition~\ref{prop:order} says that the edge polytope of $K_{r_1,r_2}$ (resp. $K_{1,r_2,r_3}$) 
is unimodularly equivalent to the order polytope of the poset $\Pi_{r_1-1,r_2-1}$ (resp. $\Pi_{r_1,r_2}'$). 
The class groups of Hibi rings are characterized in \cite{HHN}. 
Since the Hibi ring of the poset $\Pi_{m,n}$ is nothing but the Segre product of two polynomial rings (see \cite[Example 2.6]{HN}), 
an NCCR has been already constructed by \cite[Theorem 3.6]{HN}. 
Moreover, the Hibi rings whose class groups are isomorphic to $\ZZ^2$ have been intensively studied and their NCCRs have been constructed in \cite[Section 3]{N}. 
Therefore, we do not need to treat those cases. %from Section~\ref{sec:facet}. 
\end{rem}

Given integers $1 \leq r_1 \leq \cdots \leq r_n$, %with $n \geq 4$ or $n =3$ and $2 \leq r_1 \leq \cdots \leq r_3$, 
let $\calC(r_1,\ldots,r_n)$ be a convex polytope defined as follows: 
\begin{equation}\label{conic_region}
\begin{split}
%\calC(r_1,\ldots,r_n)=&\bigg\{ (z_1,\ldots,z_n) \in \RR^n : -\sum_{\substack{1 \leq j \leq n-1 \\ j \neq i}}r_j+1 \leq z_i \leq 1 \text{ for }i=1,\ldots,n-1, \\
%&-\sum_{j=1}^{n-1}r_j+1 \leq z_n \leq r_n+1, \; -r_i \leq z_j-z_i \leq r_j \text{ for }1 \leq i, j \leq n \bigg\}. 
\calC(r_1,\ldots,r_n)=&\bigg\{ (z_1,\ldots,z_n) \in \RR^n : -r_i \leq z_j-z_i \leq r_j \text{ for }1 \leq i, j \leq n, \\
&-|J|-\sum_{i \in [n-1] \setminus I}r_i-\sum_{j \in J}r_j+1 \leq \sum_{i \in I}z_i-\sum_{j \in J}z_j \leq |J|+1 \\
&\text{ for }I,J \subset [n-1] \text{ with }|I|=|J|+1 \text{ and }I \cap J = \emptyset, \\
&-|J|-\sum_{i \in [n-1] \setminus I}r_i-\sum_{j \in J}r_j+2 \leq \sum_{i \in I}z_i-\sum_{j \in J}z_j \leq |J| \\
&\text{ for }I \subset [n-1] \text{ and }J \subset [n]\text{ with }|I|=|J|+1, \; n \in J \text{ and }I \cap J = \emptyset\bigg\}, 
\end{split}
\end{equation}
where $J$ is regarded as a multi-set and $J=\emptyset$ might happen, while $I$ is a usual non-empty set. 
For the explicit descriptions in the cases where $n=3$ and $n=4$, see Example~\ref{ex:conic}. 

\begin{thm}\label{main2}
Let $K_{r_1,\ldots,r_n}$ be the complete multipartite graph with $1 \leq r_1 \leq \cdots \leq r_n$. %Assume that $n \geq 4$ or $n=3$ and $r_1 \geq 2$. 
Then the conic divisorial ideals of $\kk[K_{r_1,\ldots,r_n}]$ one-to-one correspond to the points in $\calC(r_1,\ldots,r_n) \cap \ZZ^n$ if $n \leq 4$. 
\end{thm}
%As the following remark says, the condition $n \leq 4$ appearing in Theorem~\ref{main2} is necessary for the edge ring $\kk[K_{r_1,\ldots,r_n}]$ to be Gorenstein. 
\begin{rem}\label{rem:Gor}%[{\cite[Remark 2.8]{OH00}}]
It is claimed in \cite[Remark 2.8]{OH00} that $\kk[K_{r_1,\ldots,r_n}]$ is Gorenstein if and only if 
\begin{itemize}
\item[(i)] $n=2$ with $r_1=1$ or $r_1=r_2$, or 
\item[(ii)] $n=3$ with $1 \leq r_1 \leq r_2 \leq r_3 \leq 2$, or 
\item[(iii)] $n=4$ with $r_1=\cdots=r_4=1$. 
\end{itemize}
This can be proved by using the facet descriptions of edge polytopes together with the technique employed in \cite{DH}. 
\end{rem}
%Note that the cases $n=2$ and $n=3$ with $r_1=1$ can be be deduced to Hibi rings by Proposition~\ref{prop:order}. 
%Thus, the essential cases of Theorem~\ref{main2} are $n=4$ and $n=3$ with $r_1 \geq 2$. 

It is known by \cite[Theorem 1.1]{DITW} that if a ring admits an NCCR, then it should be $\QQ$-Gorenstein. 
In particular, $\kk[K_{r_1,\ldots,r_n}]$ admitting an NCCR must be Gorenstein. 
Since Remark~\ref{rem} says that the essential cases of Theorem~\ref{main2} are $n=3$ with $r_1 \geq 2$ and $n=4$, 
for the investigation of the existence of an NCCR of $\kk[K_{r_1,\ldots,r_n}]$, 
our remaining tasks are the study of the edge rings of $K_{2,2,2}$ and $K_{1,1,1,1}=K_4$ by Remark~\ref{rem:Gor}. 
\begin{thm}\label{main3}
Let $R$ be the edge ring of $G=K_{2,2,2}$ or $G=K_4$. Let \begin{align*}\calL=
\{(0,0,-1), (0,0,0), (1,0,0), (1,-1,0),(1,0,1),(0,-1,-1),(1,-1,-1),(0,-1,-2)\}
\end{align*}
if $G=K_{2,2,2}$, and let 
\begin{align*}\calL=\{(0,0,0,0), (1,0,0,0), (1,0,0,1), (1,1,0,1), (1,1,1,2)\}\end{align*}
if $G=K_4$. Then $\displaystyle \End_R\left(\bigoplus_{\chi \in \calL} M_\chi \right)$ is an NCCR of $R$, respectively. 
\end{thm}

\subsection{Organization}
In Section~\ref{sec:poset}, we prove that the edge rings of $K_{r_1,r_2}$ and $K_{1,r_2,r_3}$ can be deduced to certain Hibi rings, respectively (Proposition~\ref{prop:order}). 
Since Hibi rings are the toric rings of certain lattice polytopes arising from posets, called order polytopes, 
and the proof of Proposition~\ref{prop:order} is based on the unimodular equivalence of order polytopes and chain polytopes for some posets (see Theorem~\ref{X}), 
we recall the notions of such poset polytopes. %(originally introduced by Stalney \cite{S86}). 
In Section~\ref{sec:facet}, for the computations of class groups and conic divisorial ideals of $\kk[K_{r_1,\ldots,r_n}]$, 
we recall the facet descriptions of edge polytopes. %(originally proved by Ohsugi--Hibi \cite{OH98}). 
In Section~\ref{sec:class}, we prove Theorem~\ref{main1}. 
In Section~\ref{sec:conic}, we prove Theorem~\ref{main2}. 
In Section~\ref{sec:NCCR}, we prove Theorem~\ref{main3}. 

%%%%%%%%%%%%%%%%%%%%%%%%%%%%%%%%%%%%%%%%%%%%%%%%%%%%%%%%%%%%%%%%%%%%%%%%%%
\subsection*{Acknowledgement} 
The authors would like to thank Yusuke Nakajima for a lot of his helpful comments on the results. 
The first named author is partially supported by JSPS Grant-in-Aid for Scientists Research (C) 20K03513. 
%%%%%%%%%%%%%%%%%%%%%%%%%%%%%%%%%%%%%%%%%%%%%%%%%%%%%%%%%%%%%%%%%%%%%%%%%%

\bigskip

%%%%%%%%%%%%%%%%%%%%%%%%%%%%%%%%%%%%%%%%%%%%%%%%%%%%%%%%%%%%%%%%%%%%%%%%%%%%%%%%%%%%%%%%%%%%%%%%%%%%%%%%%%%%%%%%%%%%%%%%%%%%%
%%%%%%%%%%%%%%%%%%%%%%%%%%%%%%%%%%%%%%%%%%%%%%%%%%%%%%%%%%%%%%%%%%%%%%%%%%%%%%%%%%%%%%%%%%%%%%%%%%%%%%%%%%%%%%%%%%%%%%%%%%%%%
%%%%%%%%%%%%%%%%%%%%%%%%%%%%%%%%%%%%%%%%%%%%%%%%%%%%%%%%%%%%%%%%%%%%%%%%%%%%%%%%%%%%%%%%%%%%%%%%%%%%%%%%%%%%%%%%%%%%%%%%%%%%%
%%%%%%%%%%%%%%%%%%%%%%%%%%%%%%%%%%%%%%%%%%%%%%%%%%%%%%%%%%%%%%%%%%%%%%%%%%%%%%%%%%%%%%%%%%%%%%%%%%%%%%%%%%%%%%%%%%%%%%%%%%%%%

\section{Two poset polytopes and Hibi rings}\label{sec:poset}
Before discussing the edge polytopes of graphs, we recall two polytopes arising from posets, called order polytopes and chain polytopes. 
The monoid $\kk$-algebras of order polytopes are called Hibi rings. 
As explained in Introduction, NCCRs of certain Hibi rings have been investigated in \cite{HN} and \cite{N}. 
We will see that some edge polytopes of complete multipartite graphs are unimodularly equivalent to some order polytopes (Proposition~\ref{prop:order}).

Let $\Pi=\{p_1,\ldots,p_{d-1}\}$ be a finite partially ordered set (poset, for short) equipped with a partial order $\prec$. 
For $p,q \in \Pi$, we say that \textit{$p$ covers $q$} if $q \prec p$ and there is no $p' \in \Pi \setminus \{p,q\}$ with $q \prec p' \prec p$. 
For a subset $I \subset \Pi$, we say that $I$ is a \textit{poset ideal} of $\Pi$ if $p \in I$ and $q \prec p$ then $q \in I$. 
For a subset $A \subset \Pi$, we call $A$ an \textit{antichain} of $\Pi$ if $p \not\prec q$ and $q \not\prec p$ for any $p,q \in A$ with $p \neq q$. 
Note that $\emptyset$ is regarded as a poset ideal and an antichain. Let \begin{align*}
\calO(\Pi)=\{(x_1,\ldots,x_{d-1}) \in \RR^{d-1} : \; x_i \geq x_j \text{ if } p_i \prec p_j \text{ in }\Pi, \;\; 
0 \leq x_i \leq 1 \text{ for }i=1,\ldots,d-1\}.
\end{align*}
A convex polytope $\calO(\Pi)$ is called the \textit{order polytope} of $\Pi$. 
It is known (\cite{S86}) that $\calO(\Pi)$ is a $(0,1)$-polytope and the vertices of $\calO(\Pi)$ one-to-one correspond to the poset ideals of $\Pi$. 
In fact, a $(0,1)$-vector $(a_1,\ldots,a_{d-1})$ is a vertex of $\calO(\Pi)$ if and only if $\{ p_i \in \Pi : a_i =1 \}$ is a poset ideal. 

%Let $\kk$ be a field. 
Given a poset $\Pi$, let $\kk[\Pi]$ be the $\kk$-algebra generated by those monomials corresponding to the lattice points in $\Pi$, i.e., 
$$\kk[\Pi]:=\kk[{\bf x}^\alpha t : \alpha \in \Pi \cap \ZZ^{d-1}],$$
where ${\bf x}^\alpha=x_1^{\alpha_1}\cdots x_{d-1}^{\alpha_{d-1}}$ for $\alpha = (\alpha_1,\ldots,\alpha_{d-1}) \in \ZZ^{d-1}$ 
and each ${\bf x}^\alpha t$ is defined to be degree $1$. 
The standard graded monoid $\kk$-algebra $\kk[\Pi]$ is called the \textit{Hibi ring} of $\Pi$. 
The following fundamental properties on Hibi rings were originally proved in \cite{H87}: 
\begin{itemize}
\item The Krull dimension of $\kk[\Pi]$ is $|\Pi|+1$; 
\item $\kk[\Pi]$ is a Cohen--Macaulay normal domain; 
\item $\kk[\Pi]$ is an algebra with straightening laws on $\Pi$. 
\end{itemize}

We also recall another polytope arising from $\Pi$, which is defined as follows: 
\begin{align*}
\calC(\Pi)=\{(x_1,\ldots,x_{d-1}) \in \RR^{d-1} : \;&x_i \geq 0 \text{ for }i=1,\ldots,d-1, \\
&x_{i_1}+\cdots+x_{i_k} \leq 1 \text{ for }p_{i_1} \prec \cdots \prec p_{i_k} \text{ in }\Pi\}.\end{align*} 
A convex polytope $\calC(\Pi)$ is called the \textit{chain polytope} of $\Pi$. 
Similarly to order polytopes, it is known (\cite{S86}) that $\calC(\Pi)$ is a $(0,1)$-polytope and the vertices of $\calC(\Pi)$ one-to-one correspond to the antichains of $\Pi$. 
%In fact, a $(0,1)$-vector $(a_1,\ldots,a_{d-1})$ is a vertex of $\calO(\Pi)$ if and only if $\{ p_i \in \Pi : a_i =1 \}$ is a poset ideal. 

\begin{thm}[{\cite[Theorem 2.1]{HL16}}]\label{X}
Let $\Pi$ be a poset. Then $\calO(\Pi)$ and $\calC(\Pi)$ are unimodularly equivalent if and only if $\Pi$ does not contain the ``X-shape'' subposet. 
\end{thm}
Here, the ``X-shape'' poset is a poset $\{z_1,z_2,z_3,z_4,z_5\}$ equipped with the partial orders $z_1 \prec z_3 \prec z_4$ and $z_2 \prec z_3 \prec z_5$. 

Given positive integers $m$ and $n$, let $\Pi_{m,n}=\{p_1,\ldots,p_m,p_{m+1},\ldots,p_{m+n}\}$ be the poset equipped with the partial orders 
$p_1 \prec \cdots \prec p_m$ and $p_{m+1} \prec \cdots \prec p_{m+n}$. 
Moreover, let $\Pi_{m,n}'$ be the poset having an additional relation $p_1 \prec p_{m+n}$. 
Note that $\Pi_{m,n}$ is the poset appearing in \cite[Example 2.6]{HN} with $t=2$, $r_1=m$ and $r_2=n$ 
and its Hibi ring $\kk[\Pi_{m,n}]$ is isomorphic to the Segre product of 
the polynomial ring with $(m+1)$ variables and the polynomial ring with $(n+1)$ variables. 

We notice that both $\Pi_{m,n}$ and $\Pi_{m,n}'$ do not contain the X-shape subposet, 
so $\calO(\Pi_{m,n})$ (resp. $\calO(\Pi_{m,n}')$) is unimodularly equivalent to $\calC(\Pi_{m,n})$ (resp. $\calC(\Pi_{m,n}')$) by Theorem~\ref{X}. 

\begin{prop}\label{prop:order} Let $m,n$ be positive integers. 
\begin{itemize}
\item[(1)] The edge polytope $P_{K_{m+1,n+1}}$ is unimodularly equivalent to the order polytope $\calO(\Pi_{m,n})$. 
\item[(2)] The edge polytope $P_{K_{1,m,n}}$ is unimodularly equivalent to the order polytope $\calO(\Pi_{m,n}')$. 
\end{itemize}
In particular, the edge ring $\kk[K_{m+1,n+1}]$ (resp. $\kk[K_{1,m,n}]$) is isomorphic to the Hibi ring $\kk[\Pi_{m,n}]$ (resp. $\kk[\Pi_{m,n}']$). 
\end{prop}
\begin{proof}
%Since $\calO(\Pi_{m,n}) \cong \calC(\Pi_{m-1,n-1})$ (resp. $\calO(\Pi_{m,n}') \cong \calC(\Pi_{m,n}')$), 
It is enough to show that $P_{K_{m+1,n+1}}$ (resp. $P_{K_{1,m,n}}$) is unimodularly equivalent to  $\calC(\Pi_{m,n})$  (resp. $\calC(\Pi_{m,n}')$). 

\medskip

\noindent
(1) Let $V(K_{m+1,n+1})=[m+n+2]$ and let $E(K_{m+1,n+1})=\{\{i,j\} : 1 \leq i \leq m+1, m+2 \leq j \leq m+n+2\}$. 
%Note that $P_{K_{m+1,n+1}}$ is of dimension $m+n$. 
Then it is straightforward to see that the vertices of $P_{K_{m+1,n+1}}$ one-to-one correspond to the antichains of $\Pi_{m,n}$ 
by consider the projection $\RR^{m+n+2} \rightarrow \RR^{m+n}$ ignoring the $(m+1)$-th and $(m+n+2)$-th coordinates 
and this projection gives a unimodular transformation between $P_{K_{m+1,n+1}}$ and $\calC(\Pi_{m,n})$. 
%In fact, the edge $\{i,j\}$ corresponds to the antichain $\{p_i,p_{j-1}\}$ if $i \neq m$ and $j \neq m+n+2$, 
%$\{p_i\}$ if $i\neq m$ and $j=m+n$, and $\{p_j\}$ if $i=m$ and $j \neq m+n$, and $\emptyset$ if $(i,j)=(m,m+n)$. 

\medskip

\noindent
(2) Let $V(K_{1,m,n})=[m+n+1]$ and let $$E(K_{1,m,n})=\{\{i,j\} : 1 \leq i \leq m, m+1 \leq j \leq m+n\} \cup \{\{k,m+n+1\} : 1 \leq k \leq m+n\}.$$ 
%Note that $P_{K_{1,m.n}}$ is of dimension $m+n$. 
Consider the projection $\RR^{m+n+1} \rightarrow \RR^{m+n}$ by ignoring the $(m+n+1)$-th coordinate. 
Then the set of vertices of $P_{K_{1,m.n}}$ becomes $\{\eb_i+\eb_j : 1 \leq i \leq m, m+1 \leq j \leq m+n\} \cup \{\eb_k : 1 \leq k \leq m+n\}$. 
By applying a unimodular transformation $\begin{pmatrix}
1 &1 &\cdots &1 &  &  &       & \\
  &1 &       &  &  &  &       & \\
  &  &\ddots &  &  &  &       & \\
  &  &       &1 &  &  &       & \\
  &  &       &  &1 &  &       & \\
  &  &       &  &  &1 &       & \\
  &  &       &  &  &  &\ddots & \\
  &  &       &  &1 &1 &\cdots &1 
\end{pmatrix}$ to those vertices (from the left-hand side) and translating them by $-\eb_1-\eb_{m+n}$ and applying a unimodular transformation $\begin{pmatrix}
-1 &  &       &  & \\
   &1 &       &  & \\
   &  &\ddots &  & \\
   &  &       &1 & \\
   &  &       &  &-1
\end{pmatrix}$, the set of vertices becomes as follows: \begin{align*}
&\eb_i+\eb_j \mapsto \eb_1+\eb_i+\eb_j+\eb_{m+n} \mapsto \eb_i+\eb_j \mapsto \eb_i+\eb_j \; (1< i \leq m, \; m+1 \leq j < m+n) \\
&\eb_i+\eb_{m+n} \mapsto \eb_1+\eb_i+\eb_{m+n} \mapsto \eb_i \mapsto \eb_i \; (1 \leq i < m) \\
&\eb_1+\eb_j \mapsto \eb_1+\eb_j+\eb_{m+n} \mapsto \eb_j \mapsto \eb_j \; (m+1 \leq j < m+n), \quad
\eb_1+\eb_{m+n} \mapsto {\bf 0} \\
&\eb_k \mapsto \eb_1+\eb_k \mapsto \eb_k-\eb_{m+n} \mapsto \eb_k+\eb_{m+n} \; (1 \leq k \leq m), \;\; 
\eb_k \mapsto \eb_1+\eb_k \;(m+1 \leq k \leq m+n). 
%
%
%\{\eb_i+\eb_j,\eb_i+\eb_{m+n},\eb_1+\eb_j : 2 \leq i \leq m, m+1 \leq j \leq m+n-1\} \cup \{\eb_k : 1 \leq k \leq m+n\} \cup \{{\bf 0}\}. 
\end{align*}
We can directly see that these lattice points one-to-one correspond to the antichains of $\Pi_{m,n}'$. 
\end{proof}

\bigskip

%%%%%%%%%%%%%%%%%%%%%%%%%%%%%%%%%%%%%%%%%%%%%%%%%%%%%%%%%%%%%%%%%%%%%%%%%%%%%%%%%%%%%%%%%%%%%%%%%%%%%%%%%%%%%%%%%%%%%%%%%%%%%
%%%%%%%%%%%%%%%%%%%%%%%%%%%%%%%%%%%%%%%%%%%%%%%%%%%%%%%%%%%%%%%%%%%%%%%%%%%%%%%%%%%%%%%%%%%%%%%%%%%%%%%%%%%%%%%%%%%%%%%%%%%%%
%%%%%%%%%%%%%%%%%%%%%%%%%%%%%%%%%%%%%%%%%%%%%%%%%%%%%%%%%%%%%%%%%%%%%%%%%%%%%%%%%%%%%%%%%%%%%%%%%%%%%%%%%%%%%%%%%%%%%%%%%%%%%
%%%%%%%%%%%%%%%%%%%%%%%%%%%%%%%%%%%%%%%%%%%%%%%%%%%%%%%%%%%%%%%%%%%%%%%%%%%%%%%%%%%%%%%%%%%%%%%%%%%%%%%%%%%%%%%%%%%%%%%%%%%%%

\section{Facets of edge polytopes}\label{sec:facet}
In this section, we recall the description of the facets of edge polytopes from \cite{OH98}. 
Let $G$ be a graph on the vertex set $V(G)=[d]$. 

First, we recall some notions and notation from graph theory. 
For a subset $W \subset V(G)$, let $G_W$ be the induced subgraph with respect to $W$. 
For a vertex $v$, we denote by $G \setminus v$ instead of $G_{V(G) \setminus \{v\}}$. 
Similarly, for $S \subset V(G)$, we denote by $G \setminus S$ instead of $G_{V(G) \setminus S}$. 
We say that $T \subset V(G)$ an \textit{independent set} (or \textit{stable set}) if $\{v,w\} \not\in E(G)$ for any two vertices $v,w \in T$. 
Given $v \in V(G)$, let $N_G(v)=\{w \in V(G) : \{v,w\} \in E(G)\}$. For $S \subset V(G)$, let $N_G(S)=\bigcup_{v \in S}N_G(v)$.

The following terminologies are used in \cite{OH98}: 
\begin{itemize}
\item We call a vertex $v$ of $G$ \textit{regular} if each connected component of $G \setminus v$ contains an odd cycle. 
\item Given an independent set $T \subset V(G)$, let $B(T)$ denote the bipartite graph on $T \cup N_G(T)$ with the edge set $\{\{v,w\} : v \in T, w \in N_G(T)\} \cap E(G)$. 
\item A nonempty $T \subset V(G)$ is said to be a \textit{fundamental set} if the following conditions are satisfied: 
\begin{itemize}
\item $B(T)$ is connected; 
\item $V(B(T))=V(G)$, or each connected component of $G \setminus T \cup N_G(T)$ contains an odd cycle. 
\end{itemize}
\end{itemize}

Given $i \in [d]$, let 
$$\calH_i=\{(x_1,\ldots,x_d) \in \RR^d : x_i=0\} \;\text{ and }\;\calH_i^{(+)}=\{(x_1,\ldots,x_d) \in \RR^d : x_i \geq 0\}.$$ 
Given $T \subset [d]$, let 
\begin{align*}
\calH_T&=\left\{(x_1,\ldots,x_d) \in \RR^d : \sum_{j \in N_G(T)}x_j - \sum_{i \in T}x_i = 0\right\} \;\text{ and }\\
\calH_T^{(+)}&=\left\{(x_1,\ldots,x_d) \in \RR^d : \sum_{j \in N_G(T)}x_j - \sum_{i \in T}x_i \geq 0\right\}.
\end{align*}
It is proved in \cite[Theorem 1.7 (a)]{OH98} that for any non-bipartite graph $G$, each facet of $P_G$ is defined by 
a supporting hyperplane $H_i$ for some regular vertex $i$ or $H_T$ for some fundamental set. Let 
%\begin{align*}
%P_G=\bigcap_{\text{$i$ : regular vertices}}H_i^{(+)} \cap \bigcap_{\text{$T$ : fundamental sets}}H_T^{(+)}
%\end{align*}
$$\widetilde{\Psi}=\{H_i : \text{$i$ is a regular vertex}\} \cup \{ H_T : \text{$T$ is a fundamental set}\}.$$ 
Although $\widetilde{\Psi}$ describes all supporting hyperplanes of the facets of $P_G$, 
it might happen that $H_i$ and $H_T$ define the same facet for some $i$ and $T$. 
Each hyperplane in $\widetilde{\Psi}$ can be identified with a linear form in $(\RR^d)^* \cong \RR^d$ as follows. 
Let $\langle \cdot,\cdot \rangle : (\RR^d)^* \times \RR^d \rightarrow \RR$ be a natural pairing. 
Note that the linear form which gives a hyperplane in $\widetilde{\Psi}$ is not uniquely determined, 
but we can define a unique linear form $\ell_H \in \QQ^d$ for each hyperplane $H \in \widetilde{\Psi}$ with the following condition: 
\begin{itemize}
\item[(i)] $\langle \ell_H, \alpha \rangle \in \ZZ$ for any $\alpha \in P_G \cap \ZZ^d$; \quad
(ii) $\sum_{\alpha \in P_G \cap \ZZ^d}\langle \ell_H,\alpha\rangle\ZZ=\ZZ$. 
\end{itemize}

Let $\Psi=\Psi_r \cup \Psi_f$, where \begin{align*}
\Psi_r=\{\ell_{H_i} : i \text{ is a regular vertex}\} \text{ and } \Psi_f=\{\ell_{H_T} : T \text{ is a fundamental set}\}. 
\end{align*}

\begin{ex}\label{ex1}
Consider $G=K_{2,2,2}$. Let $V(G)=[6]=\{1,2\} \cup \{3,4\} \cup \{5,6\}$ and 
\begin{align}\label{edge_order}
E(G)=\{\{1,3\},\{1,4\},\{1,5\},\{1,6\},\{2,3\},\{2,4\},\{2,5\},\{2,6\},\{3,5\},\{3,6\},\{4,5\},\{4,6\}\}. 
\end{align}
%$$E(G)=\{\{i,j\},\{i,k\},\{j,k\} : i \in \{1,2\}, j \in \{3,4\}, k \in \{5,6\}\}.$$
We see that each $i \in [6]$ is regular, and $T \subset V(G)$ is a fundamental set if and only if $T=\{i,i+1\}$ for $i=1,3,5$. 
\begin{itemize}
\item For each regular vertex $i \in [6]$, we have $\ell_{H_i}=\eb_i \in \RR^6$. 
\item For each fundamental set $T=\{i,i+1\}$, we have $\ell_{H_T}=\frac{1}{2}(\sum_{p \in [6]\setminus T}\eb_p - \sum_{q \in T}\eb_q)$. 
\end{itemize}
Thus, $\Psi$ consists of the column vectors of the following matrix: 
$$\begin{pmatrix}
1 &0 &0 &0 &0 &0 &-\frac{1}{2} &\frac{1}{2} &\frac{1}{2} \\
0 &1 &0 &0 &0 &0 &-\frac{1}{2} &\frac{1}{2} &\frac{1}{2} \\
0 &0 &1 &0 &0 &0 &\frac{1}{2} &-\frac{1}{2} &\frac{1}{2} \\
0 &0 &0 &1 &0 &0 &\frac{1}{2} &-\frac{1}{2} &\frac{1}{2} \\
0 &0 &0 &0 &1 &0 &\frac{1}{2} &\frac{1}{2} &-\frac{1}{2} \\
0 &0 &0 &0 &0 &1 &\frac{1}{2} &\frac{1}{2} &-\frac{1}{2} \\
\end{pmatrix}.$$
\end{ex}

\bigskip

%%%%%%%%%%%%%%%%%%%%%%%%%%%%%%%%%%%%%%%%%%%%%%%%%%%%%%%%%%%%%%%%%%%%%%%%%%%%%%%%%%%%%%%%%%%%%%%%%%%%%%%%%%%%%%%%%%%%%%%%%%%%%
%%%%%%%%%%%%%%%%%%%%%%%%%%%%%%%%%%%%%%%%%%%%%%%%%%%%%%%%%%%%%%%%%%%%%%%%%%%%%%%%%%%%%%%%%%%%%%%%%%%%%%%%%%%%%%%%%%%%%%%%%%%%%
%%%%%%%%%%%%%%%%%%%%%%%%%%%%%%%%%%%%%%%%%%%%%%%%%%%%%%%%%%%%%%%%%%%%%%%%%%%%%%%%%%%%%%%%%%%%%%%%%%%%%%%%%%%%%%%%%%%%%%%%%%%%%
%%%%%%%%%%%%%%%%%%%%%%%%%%%%%%%%%%%%%%%%%%%%%%%%%%%%%%%%%%%%%%%%%%%%%%%%%%%%%%%%%%%%%%%%%%%%%%%%%%%%%%%%%%%%%%%%%%%%%%%%%%%%%
%%%%%%%%%%%%%%%%%%%%%%%%%%%%%%%%%%%%%%%%%%%%%%%%%%%%%%%%%%%%%%%%%%%%%%%%%%%%%%%%%%%%%%%%%%%%%%%%%%%%%%%%%%%%%%%%%%%%%%%%%%%%%
%%%%%%%%%%%%%%%%%%%%%%%%%%%%%%%%%%%%%%%%%%%%%%%%%%%%%%%%%%%%%%%%%%%%%%%%%%%%%%%%%%%%%%%%%%%%%%%%%%%%%%%%%%%%%%%%%%%%%%%%%%%%%

%\section{Proof of Theorem~\ref{main1}}\label{sec:class}
\section{Class groups of edge rings of complete multipartite graphs}\label{sec:class}

In this section, we give a proof of Theorem~\ref{main1}. Namely, we compute the class groups of $\kk[K_{r_1,\ldots,r_n}]$. 
For this, we recall the general description of the class groups of monoid $\kk$-algebras (\cite[Theorem 9.8.19]{Villa}). 
We use the theory in \cite[Section 9.8]{Villa} and apply it to get the class group $\Cl(\kk[G])$ of $\kk[G]$.

Let $G$ be a graph. 
Given $\alpha \in P_G \cap \ZZ^d$, we define $w_\alpha$ belonging to a free abelian group $\bigoplus_{\ell \in \Psi} \ZZ \eb_\ell$ 
with its basis $\{\eb_\ell\}_{\ell \in \Psi}$ as follows: 
$$w_\alpha=\sum_{\ell \in \Psi}\langle \ell,\alpha \rangle \eb_\ell.$$
Let $\calM$ be the matrix whose column vectors consist of $w_\alpha$ for $\alpha \in P_G \cap \ZZ^d$. 
%Note that $P_G \cap \ZZ^d$ one-to-one corresponds to $E(G)$ by definition. 
%In \cite[Section 9.8]{Villa}, the divisor class groups of monoid algebras are discussed. By using those theories, we see the following: 
%
\begin{prop}[cf. {\cite[Theorem 9.8.19]{Villa}}]\label{prop:class_group}
Let $G$ be a non-bipartite graph such that $\kk[G]$ is normal and assume that $\Psi$ is irredundant. Then 
$$\Cl(\kk[G]) \cong \bigoplus_{\ell \in \Psi} \ZZ \eb_\ell \big/ \sum_{\alpha \in P_G \cap \ZZ^d} \ZZ w_\alpha.$$ 
In particular, we have 
$$\Cl(\kk[G]) \cong \ZZ^t \oplus \ZZ/d_1\ZZ \oplus \cdots \oplus \ZZ/d_s\ZZ,$$ 
where $t=|\Psi|-\rank \calM$ and $d_1,\ldots,d_s$ are positive integers appearing in the diagonal of the Smith normal form of $\calM$. 
\end{prop}
It is proved in \cite{OH98} and \cite{SVV} that the edge ring $\kk[G]$ is normal if and only if $G$ satisfies odd cycle condition, 
where we say that $G$ satisfies {\em odd cycle condition} if for each pair of odd cycles $C$ and $C'$ with no common vertex, 
there is an edge $\{v,v'\}$ with $v \in V(C)$ and $v' \in V(C')$. 
\begin{ex}\label{ex2}
Let us consider $G=K_{2,2,2}$ again. Then $|\Psi|=9$ as we saw in Example~\ref{ex1}. 
One can compute $\calM$ as follows: 
\begin{align}\label{calM}
\calM=\left(
\begin{array}{rrrrrrrrrrrr}
1 &1 &1 &1 &0 &0 &0 &0 &0 &0 &0 &0 \\
0 &0 &0 &0 &1 &1 &1 &1 &0 &0 &0 &0 \\
1 &0 &0 &0 &1 &0 &0 &0 &1 &1 &0 &0 \\
0 &1 &0 &0 &0 &1 &0 &0 &0 &0 &1 &1 \\
0 &0 &1 &0 &0 &0 &1 &0 &1 &0 &1 &0 \\
0 &0 &0 &1 &0 &0 &0 &1 &0 &1 &0 &1 \\
0 &0 &0 &0 &0 &0 &0 &0 &1 &1 &1 &1 \\
0 &0 &1 &1 &0 &0 &1 &1 &0 &0 &0 &0 \\
1 &1 &0 &0 &1 &1 &0 &0 &0 &0 &0 &0 
\end{array}
\right).
\end{align}
The columns are labeled by $P_G \cap \ZZ^d$, i.e., labeled by $E(G)$ in the ordering of \eqref{edge_order}. 
A direct computation implies that this matrix has rank $6$ and all diagonals of its Smith normal form are $1$. 
Hence, we conclude that $\Cl(\kk[G]) \cong \ZZ^{9-6}=\ZZ^3$. 
\end{ex}

Now, we are in the position to give a proof of Theorem~\ref{main1}. 

\begin{proof}[Proof of Theorem~\ref{main1}]
Let $G=K_{r_1,\ldots,r_n}$ on the vertex set $\bigsqcup_{i=1}^nV_i$ with the edge set 
$\{\{a,b\} : a \in V_i, b \in V_j \text{ for }1 \leq i \neq j \leq n\}$. Then we see that $G$ satisfies odd cycle condition. 
Hence, we can apply Proposition~\ref{prop:class_group}. 
Let $V_i=\{v_{i1},\ldots,v_{ir_i}\}$ for $i=1,\ldots,n$ and let $d=|V(G)|=\sum_{i=1}^n r_i$.

We see that any vertex $v$ in $G$ is regular. 
Moreover, we also see that $T \subset V(G)$ is a fundamental set if and only if $T=V_i$ for each $i=1,\ldots,n$. 
Hence, $\Psi$ is as follows: 
%the system of the supporting hyperplanes of $P_G$ is as follows:
\begin{equation}\label{eq:hyp_P_G}
%\begin{split}
%&x_i \geq 0 \;\text{ for }\;i=1,\ldots,d, \\
%&\sum_{k \in V(G) \setminus V_i}x_k - \sum_{j \in V_i}x_j \geq 0 \;\text{ for }\;i=1,\ldots,n, 
%\end{split}
\Psi_r = \{\eb_i :i=1,\ldots,d\} \;\text{ and }\;
\Psi_f=\left\{\frac{1}{2}\left(\sum_{k \in V(G) \setminus V_i}\eb_k - \sum_{j \in V_i}\eb_j\right) : i=1,\ldots,n \right\}. 
\end{equation}
(See Example~\ref{ex1} in the case $n=3$ with $r_1=r_2=r_3=2$.) 
Furthermore, it follows that $\widetilde{\Psi}$ is irredundant. 

Take any edge $e=\{a,b\} \in V_i \times V_j$. Then we see that 
\begin{align}\label{eq:description}
w_{\rho(e)}=\eb_{\ell_a}+\eb_{\ell_b}+\sum_{\ell \in \Psi_f \setminus \{\ell_{V_a},\ell_{V_b}\}}\eb_\ell
\end{align}
holds. Note that $\calM$ consists of column vectors $w_{\rho(e)}$. By Proposition~\ref{prop:class_group}, it suffices to show that 
the Smith normal form of $\calM$ is of the form whose diagonals are $d$ $1$'s, in particular, $\rank \calM=d$. 
Once we know this, the assertion holds by $|\Psi|-\rank \calM=(d+n)-d=n$. 

\medskip

In what follows, we divide our proof into three steps (a), (b) and (c): 
\begin{itemize}
\item[(a)] First, we find $d$ linearly independent column vectors of $\calM$. 
\item[(b)] Next, we show that such vectors are maximal one. 
\item[(c)] Finally, we show that all diagonals of the Smith normal form of $\calM$ are $1$. 
\end{itemize}

\medskip

\noindent
(a) Let $n$ be odd. Consider the following $d=nr_1+\sum_{i=2}^n(r_i-r_1)$ edges: 
\begin{equation}\label{eq:vectors}
\begin{split}
&\{v_{1j},v_{2j}\}, \{v_{2j},v_{3j}\}, \ldots, \{v_{n-1j},v_{nj}\}, \{v_{nj},v_{1j}\} \;\text{ for }\;j=1,\ldots,r_1, \\
&\{v_{11},v_{ij}\} \;\text{ for }\;i=2,\ldots,n \text{ and }j=r_1+1,\ldots,r_i. 
\end{split}
\end{equation}
(Remark that $r_1 \leq \cdots \leq r_n$.) Let $C_j=(v_{1j},v_{2j},\ldots,v_{nj})$ be the cycle of length $n$ for $j=1,\ldots,r_1$. 
%In what follows, we will show that $w_{\rho(e)}$'s for those edges $e$ are linearly independent and those are maximal. 
Since $n$ is odd, we see that $\rho(e)$'s for $e \in E(C_j)$ are linearly independent. 
We notice that the column vectors of $\calM$ restricted to the rows corresponding to $\Psi_r$ precisely correspond to $P_G \cap \ZZ^d$. 
Moreover, $v_{ij}$'s for $i=2,\ldots,n$ and $j=r_1+1,\ldots,r_i$ appear only once among the edges of \eqref{eq:vectors}. 
Therefore, we see that the column vectors of $\calM$ indexed by the edges in \eqref{eq:vectors} are also linearly independent. 

\medskip

In the case $n$ is even, we may take the following $d=(n-1)r_1+\sum_{i=2}^{n-1}(r_i-r_1)+r_n$ edges instead of \eqref{eq:vectors}: 
\begin{equation*}%\label{eq:vectors}
\begin{split}
&\{v_{1j},v_{2j}\}, \{v_{2j},v_{3j}\}, \ldots, \{v_{n-2j},v_{n-1j}\}, \{v_{n-1j},v_{1j}\} \;\text{ for }\;j=1,\ldots,r_1, \\
&\{v_{11},v_{ij}\} \;\text{ for }\;i=2,\ldots,n-1 \text{ and }j=r_1+1,\ldots,r_i, \\
&\{v_{11},v_{ni}\} \;\text{ for }\;i=1,\ldots,r_n. 
\end{split}
\end{equation*}
The similar discussions to the above ones can be applied to those edges. 

\medskip

\noindent
(b) Next, we prove that $\rank \calM=d$. 
Namely, it suffices to show that linearly dependent vectors appear once we add one more edge $e=\{v_{ij},v_{k\ell}\}$ into \eqref{eq:vectors}. 
Since the case $n$ is even is quite similar, we discuss only the case $n$ is odd. 

We divide $e$ into the following five cases: 
\begin{itemize}
\item[(i)] $e$ is a chord of $C_i$; 
\item[(ii)] $e$ is a bridge between $C_i$ and $C_j$; 
\item[(iii)] $e$ is an edge between a vertex of some $C_1 \setminus v_{11}$ and a vertex other than $V(C_i)$; 
\item[(iv)] $e$ is an edge between a vertex of some $C_i$ with $i \neq 1$ and a vertex other than $V(C_i)$; 
\item[(v)] $e$ is an edge between other vertices. 
\end{itemize}
Lemma~\ref{lem:linear_depend} below implies that linearly depndent vectors appear for any cases (i)--(v) as follows.
\begin{itemize}
\item[(i)] Since $C_i$ is an odd cycle, we see that an even cycle appears by adding a chord. 
Hence, the linearly dependent vectors appear by Lemma~\ref{lem:linear_depend} (a). 
\item[(ii)] In this case, the assertion holds by Lemma~\ref{lem:linear_depend} (b-2). 
\item[(iii)] Let, say, $e=\{v_{21},v_{2,r_1+1}\}$. %v21 ... vn1 v11 v2r_1+1 v21
Then $(v_{21},v_{31},\ldots,v_{n1},v_{11},v_{2,r_1+1})$ forms an even cycle. 
\item[(iv)] Let, say, $e=\{v_{12},v_{2,r_1+1}\}$. 
Then $(v_{11},v_{2,r_1+1},v_{12})$ forms a path connecting $C_1$ and $C_2$. 
Hence, the assertion follows by Lemma~\ref{lem:linear_depend} (b-2). 
%C_1 + path v_11v_{2r_1+1}v_12 + C_2
\item[(v)] Let, say, $e=\{v_{2,r_1+1},v_{3,r_1+1}\}$. 
Since $(v_{11},v_{2,r_1+1},v_{3,r_1+1})$ forms an odd cycle which shares a unique common vertex $v_{11}$ with $C_1$, 
we obtain the assertion by Lemma~\ref{lem:linear_depend} (b-1). 
\end{itemize}
%C_1 + v_11 v v' v_11 

\medskip

\noindent
(c) Finally, we compute the Smith normal form of $\calM$. Let $n$ be odd. (The case $n$ is even is similar.) 
For the computation, we add the following edges to \eqref{eq:vectors}: 
\begin{align*}
\{v_{11},v_{22}\},\{v_{12},v_{23}\},\ldots,\{v_{1,r_1-1},v_{2,r_1}\}. 
\end{align*}
Let $\calM'$ be the $(d+n) \times (d+r_1-1)$ submatrix of $\calM$ consisting of the columns corresponding to \eqref{eq:vectors} and those additional edges. 
Since we know that $\rank \calM=d$ by the steps (a) and (b), it is enough to show that the Smith normal form of $\calM'$ has $d$ $1$'s as diagonals. 

For the columns corresponding to the edges $\{v_{11},v_{ij}\}$ for $i=2,\ldots,n$ and $j=r_1+1,\ldots,r_i$, 
since $1$ in the $v_{ij}$-th row appears only once in $\calM'$, we can create $\sum_{i=2}^n(r_i-r_1)$ unit vectors. 
Let $\calM''$ be the $(d+n) \times (nr_1+r_1-1)$ submatrix of $\calM'$ whose columns are indexed by the edges of $C_i$'s and the additional edges. 
Our work is to show that the Smith normal form of $\calM''$ is of the form that there are $nr_1$ $1$'s in the diagonal. 
Note that we already know that $\rank \calM'' = nr_1$. 

We apply the following row and column operations to $\calM''$: 
\begin{enumerate}
\item Add the $(-1)\cdot$($\{v_{11},v_{21}\}$-th column) (i.e. $(-1) \cdot$(the first column of $C_1$)) to the $\{v_{11},v_{22}\}$-column, 
and add the $v_{21}$-th row (i.e. the second row of $C_1$) to the $v_{22}$-th row (i.e. the second row of $C_2$). 
Then the first and second (i.e. $\{v_{11},v_{21}\}$-th and $\{v_{21},v_{31}\}$-th) entries of the $v_{21}$-th row are $1$, 
its $\{v_{11},v_{22}\}$-th entry is $-1$, and the other entries are all $0$. 
Note that the $v_{21}$-th entry of $\{v_{11},v_{22}\}$-th column is $-1$ and the other entries are all $0$. 
Thus, we can erase the other entries of $v_{21}$-th row without changing any other entries. 
\item Apply the following operations: 
\begin{itemize}
\item[(2-1)] Add the $(-1) \cdot (\{v_{22},v_{32}\}$-th column), $\{v_{32},v_{42}\}$-th column,..., and $(-1) \cdot (\{v_{n-1,2},v_{n2}\}$-th column) 
to the $\{v_{11},v_{21}\}$-th column; 
\item[(2-2)] Add the $(-1) \cdot (\{v_{11},v_{21}\}$-th column) to the $\{v_{n1},v_{11}\}$-th column; 
\item[(2-3)] Add the $(-1) \cdot (\{v_{n1},v_{11}\}$-th column) to the $\{v_{n-1,1},v_{n1}\}$-th column, 
add the $(-1) \cdot (\{v_{n-1,1},v_{n1}\}$-th column) to the $\{v_{n-2,1},v_{n-1,1}\}$-th column,...,and 
add the $(-1) \cdot (\{v_{41},v_{51}\}$-th column) to the $\{v_{31},v_{41}\}$-th column. 
%Namely, add the columns of $C_2$ from the second one to the $(n-1)$-th one alternatingly to the first column of $C_1$, 
%and add the $-1$ times the first column of $C_1$ to the last column of $C_1$. 
\end{itemize}

Then the rows corresponding to $v_{11},\ldots,v_{n1}$ (i.e. the vertices in $C_1$) contain exactly one $1$ or $-1$. 
Thus, we can erase all other entries. 
\item Note that the operations (1) and (2) do not change the column vectors except for the ones corresponding to the edges of $C_1$ and $\{v_{11},v_{22}\}$. 
Thus, in general, by using the edges of $C_i$ and $\{v_{1i},v_{2,i+1}\}$, we can do the same operations as above for each $i=1,\ldots,r_1-1$. 
After those applications, $n \cdot (r_1-1)$ $1$'s can appear in the diagonal by arranging the rows and columns. 
\item Finally, we consider $C_{r_1}$. 
By using the $n$ rows corresponding to $\Psi_f$, we can erase the nonzero entries corresponding $v_{r_11},\ldots,v_{r_1n} \in V(C_{r_1})$. 
Then $nr_1$ $1$'s appear in the diagonal by changing the rows. By $\rank \calM'' = nr_1$, we obtain the desired Smith normal form. 
\end{enumerate}
\end{proof}
\begin{ex}
Consider the matrix \eqref{calM}. The part (c) in the above proof takes the following left-most submatrix: 
\begin{align*}
\begin{pmatrix}
1 &0 &1 &0 &0 &0 &1\\
0 &0 &0 &1 &0 &1 &0\\
1 &1 &0 &0 &0 &0 &0\\
0 &0 &0 &1 &1 &0 &1\\
0 &1 &1 &0 &0 &0 &0\\
0 &0 &0 &0 &1 &1 &0\\
0 &1 &0 &0 &1 &0 &0\\
0 &0 &1 &0 &0 &1 &0\\
1 &0 &0 &1 &0 &0 &1
\end{pmatrix}
\longrightarrow
\begin{pmatrix}
1 &0 &1 &0 &0 &0 &1 \\
1 &1 &0 &0 &0 &0 &0 \\
0 &1 &1 &0 &0 &0 &0 \\
0 &0 &0 &1 &0 &1 &0 \\
0 &0 &0 &1 &1 &0 &1 \\
0 &0 &0 &0 &1 &1 &0 \\
1 &0 &0 &1 &0 &0 &1 \\
0 &1 &0 &0 &1 &0 &0 \\
0 &0 &1 &0 &0 &1 &0 
\end{pmatrix}
\longrightarrow
\begin{pmatrix}
1 &0 &1 &0 &0 &0 &0 \\
0 &0 &0 &0 &0 &0 &-1 \\
0 &1 &1 &0 &0 &0 &0 \\
0 &0 &0 &1 &0 &1 &0 \\
1 &1 &0 &1 &1 &0 &0 \\
0 &0 &0 &0 &1 &1 &0 \\
1 &0 &0 &1 &0 &0 &0 \\
0 &1 &0 &0 &1 &0 &0 \\
0 &0 &1 &0 &0 &1 &0 
\end{pmatrix}
\end{align*}
Note that in the first matrix, 
the the first, second and third (resp. fourth, fifth and sixth) columns correspond to the odd cycle $C_1$ (resp. $C_2$), 
and the seventh column corresponds to the additional edge $\{v_{11},v_{22}\}$. 
By arranging the rows, we obtain the second matrix. 
The third one is the matrix just after the application of (1). 
\begin{align*}
\longrightarrow
\begin{pmatrix}
1 &0 &0 &0 &0 &0 &0 \\
0 &0 &0 &0 &0 &0 &-1 \\
0 &0 &1 &0 &0 &0 &0 \\
0 &0 &0 &1 &0 &1 &0 \\
0 &1 &0 &1 &1 &0 &0 \\
0 &-1&0 &0 &1 &1 &0 \\
0 &1 &0 &1 &0 &0 &0 \\
0 &0 &0 &0 &1 &0 &0 \\
0 &-1&0 &0 &0 &1 &0 
\end{pmatrix}
\longrightarrow
\begin{pmatrix}
1 &0 &0 &0 &0 &0 &0 \\
0 &1 &0 &0 &0 &0 &0 \\
0 &0 &1 &0 &0 &0 &0 \\
0 &0 &0 &1 &0 &1 &0 \\
0 &0 &0 &1 &1 &0 &1 \\
0 &0 &0 &0 &1 &1 &-1 \\
0 &0 &0 &1 &0 &0 &1 \\
0 &0 &0 &0 &1 &0 &0 \\
0 &0 &0 &0 &0 &1 &-1 
\end{pmatrix}
\longrightarrow
\begin{pmatrix}
1 &0 &0 &0 &0 &0 &0 \\
0 &1 &0 &0 &0 &0 &0 \\
0 &0 &1 &0 &0 &0 &0 \\
0 &0 &0 &0 &0 &0 &0 \\
0 &0 &0 &0 &0 &0 &0 \\
0 &0 &0 &0 &0 &0 &0 \\
0 &0 &0 &1 &0 &0 &1 \\
0 &0 &0 &0 &1 &0 &0 \\
0 &0 &0 &0 &0 &1 &-1 
\end{pmatrix}
\end{align*}
The fourth matrix is the result after the applications of (2-1), (2-2) and (2-3). 
Note that the first three rows contain the vectors consisting of only one $\pm 1$. 
By multiplying $(-1)$ to the last column and exchanging the second and the last columns, we obtain the fifth matrix. 
Finally, we obtain the last matrix by applying the operation (4). 
\end{ex}

\begin{lem}\label{lem:linear_depend}
{\em (a)} Let $e_1,\ldots,e_{2k}$ be the edges of an even cycle in $K_{r_1,\ldots,r_n}$. 
Then $$w_{\rho(e_1)},\ldots,w_{\rho(e_{2k})}$$ are linearly dependent. \\
{\em (b)} Let $C$ and $C'$ be two odd cycles and let $e_1,\ldots,e_{2k+1}$ (resp. $e_1',\ldots,e_{2k'+1}'$) be the edges of $C$ (resp. $C'$). 
\begin{itemize}
\item[(b-1)] Assume that $C$ and $C'$ have a unique common vertex. Then 
$$w_{\rho(e_1)},\ldots,w_{\rho(e_{2k+1})},w_{\rho(e_1')},\ldots,w_{\rho(e_{2k'+1}')}$$ are linearly dependent. 
\item[(b-2)] Assume that $C$ and $C'$ have no common vertex but there is a path whose edges are $f_1,\ldots,f_m$ between $C$ and $C'$ connecting them. 
Then $$w_{\rho(e_1)},\ldots,w_{\rho(e_{2k+1})},w_{\rho(e_1')},\ldots,w_{\rho(e_{2k'+1}')},w_{\rho(f_1)},\ldots,w_{\rho(f_m)}$$ 
are linearly dependent. 
\end{itemize}
\end{lem}
\begin{proof}
(a) By \eqref{eq:description}, we see that $\sum_{i=1}^{2k}(-1)^iw_{\rho(e_i)}={\bf 0}$. 

\noindent
(b) In the case (b-1), let $e_1 \cap e_{2k+1} \cap e_1' \cap e_{2k'+1}'$ be the unique common vertex of $C$ and $C'$. 
In the case (b-2), let $P$ be the path connecting the vertex $e_1 \cap e_{2k+1}$ of $C$ and $e_1' \cap e_{2k'+1}'$ of $C'$. 
Then we see the following: 
\begin{align*}
&\sum_{i=1}^{2k+1}(-1)^iw_{\rho(e_i)}-\sum_{i=1}^{2k'+1}(-1)^iw_{\rho(e_i')}={\bf 0}; \\
&\sum_{i=1}^{2k+1}(-1)^iw_{\rho(e_i)}-\sum_{i=1}^{2k'+1}(-1)^iw_{\rho(e_i')}-2\sum_{j=1}^m(-1)^jw_{\rho(f_j)}={\bf 0} \text{ if $m$ is even}; \\
&\sum_{i=1}^{2k+1}(-1)^iw_{\rho(e_i)}+\sum_{i=1}^{2k'+1}(-1)^iw_{\rho(e_i')}-2\sum_{j=1}^m(-1)^jw_{\rho(f_j)}={\bf 0} \text{ if $m$ is odd}. 
\end{align*}
\end{proof}

\bigskip

%%%%%%%%%%%%%%%%%%%%%%%%%%%%%%%%%%%%%%%%%%%%%%%%%%%%%%%%%%%%%%%%%%%%%%%%%%%%%%%%%%%%%%%%%%%%%%%%%%%%%%%%%%%%%%%%%%%%%%%%%%%%%
%%%%%%%%%%%%%%%%%%%%%%%%%%%%%%%%%%%%%%%%%%%%%%%%%%%%%%%%%%%%%%%%%%%%%%%%%%%%%%%%%%%%%%%%%%%%%%%%%%%%%%%%%%%%%%%%%%%%%%%%%%%%%
%%%%%%%%%%%%%%%%%%%%%%%%%%%%%%%%%%%%%%%%%%%%%%%%%%%%%%%%%%%%%%%%%%%%%%%%%%%%%%%%%%%%%%%%%%%%%%%%%%%%%%%%%%%%%%%%%%%%%%%%%%%%%
%%%%%%%%%%%%%%%%%%%%%%%%%%%%%%%%%%%%%%%%%%%%%%%%%%%%%%%%%%%%%%%%%%%%%%%%%%%%%%%%%%%%%%%%%%%%%%%%%%%%%%%%%%%%%%%%%%%%%%%%%%%%%
%%%%%%%%%%%%%%%%%%%%%%%%%%%%%%%%%%%%%%%%%%%%%%%%%%%%%%%%%%%%%%%%%%%%%%%%%%%%%%%%%%%%%%%%%%%%%%%%%%%%%%%%%%%%%%%%%%%%%%%%%%%%%
%%%%%%%%%%%%%%%%%%%%%%%%%%%%%%%%%%%%%%%%%%%%%%%%%%%%%%%%%%%%%%%%%%%%%%%%%%%%%%%%%%%%%%%%%%%%%%%%%%%%%%%%%%%%%%%%%%%%%%%%%%%%%

\section{Conic divisorial ideals of edge rings of complete multipartite graphs}\label{sec:conic}

In this section, we give a description of conic divisorial ideals of $\kk[K_{r_1,\ldots,r_n}]$ for $n =3,4$. 
\subsection{Preliminaries on conic divisorial ideals of toric rings} 
First, we review some basic facts on conic divisorial ideals of toric rings. 

Let $C \subset \RR^d$ be a finitely generated pointed cone defined by half-open spaces $H_i^{(+)} \subset \RR^d$ for $i=1,\ldots,m$, 
where $H_i^{(+)}=\{\xb \in \RR^d : \langle \tau_i,\xb \rangle \geq 0\}$ for some linear form $\tau_i \in (\RR^d)^*$. 
We set $\tau(-) : \RR^d \rightarrow \RR^m$ by $\tau(\xb):=(\langle \tau_1,\xb \rangle, \ldots, \langle \tau_m,\xb \rangle)$. 
We define a monoid $\kk$-algebra $R$ by setting 
$$R:=\kk[C \cap \ZZ^d]=\kk[\tb^\alpha : \alpha \in C \cap \ZZ^d] \subset \kk[t_1,\ldots,t_d],$$ 
where $\tb^\alpha=t_1^{\alpha_1}\cdots t_d^{\alpha_d}$ for each $\alpha=(\alpha_1,\ldots,\alpha_d)$. 

%We recall how to determine what kind of elements in $\Cl(R)$ corresponds to a conic divisorial ideals. 
%We review some basic facts on conic divisorial ideals of a toric ring $R$. 
Given $\ab=(a_1,\ldots,a_m) \in \RR^m$, we define 
the $R$-module $T(\ab)$ generated by the Laurent monomials whose exponents are in $\{\xb \in \ZZ^d : \tau(\xb) \geq \ab\}$, 
where $\geq$ stands for the component-wise inequality. It is known that there is an exact sequence 
\begin{equation}\label{cl_seq}
0 \longrightarrow\ZZ^d\xrightarrow{\tau(-)}\ZZ^m \longrightarrow\Cl(R) \longrightarrow 0, 
\end{equation}
where $\Cl(R)$ denotes the class group of $R$. 
Hence, we see that for $\ab, \ab^\prime\in\ZZ^m$, the divisorial ideals $T(\ab)$ and $T(\ab^\prime)$ are isomorphic as $R$-modules 
if and only if there exists $\yb \in \ZZ^d$ such that $\ab=\ab^\prime + \tau(\yb)$ (see, e.g., \cite[Corollary~4.56]{BG2}).

%
%
%In addition, the exact sequence (\ref{cl_seq}) gives the relations on these divisors. 
%Precisely, a divisor $a_1\calD_1+\cdots +a_n\calD_n$ is zero in $\Cl(R)$ if $a_i=\langle \xb,v_1\rangle$ for all $i=1,\cdots,n$ and some $\xb\in\sfM$. 
%In particular, if we take the $j$-th basic vector $\eb_j:=(\delta_{j1},\cdots,\delta_{jd})\in\sfM$, 
%we have that $$\langle\eb_j,v_1\rangle D_1+\cdots+\langle\eb_j,v_n\rangle D_n=0$$ in $\Cl(R)$ for all $j=1,\cdots,d$. 
%Thus, we have that 
%\begin{equation}
%\label{relation_divisor}
%v_{1,j}\calD_1+\cdots+v_{n,j}\calD_n=0 
%\end{equation} 
%for all $j=1,\cdots,d$ where $v_i:=(v_{i,1},\cdots,v_{i,d})\in\ZZ^d$. 
%

\begin{defi}[{See, e.g., \cite[Section 3]{BG2}}]
A divisorial ideal $T(\ab)$ is said to be \textit{conic} if there is $\xb \in \RR^d$ with $\ab=\lceil\tau(\xb)\rceil$, 
where $\lceil \cdot \rceil$ denotes the ceiling function and $\lceil (u_1,\ldots,u_m) \rceil =(\lceil u_1 \rceil, \ldots, \lceil u_m \rceil)$. 
In other words, there is $\xb \in \RR^d$ such that $\ab-{\bf 1} < \tau(\xb) \leq \ab$, where ${\bf 1}=(1,1,\ldots,1)$. 
\end{defi}
Note that a conic divisorial ideal is determined by the elements in $\RR^d/\ZZ^d$ up to isomorphism 
since we see that $T(\tau(\xb^\prime))\cong T(\tau(\xb))$ for $\xb,\xb^\prime \in \RR^d$ with $\xb^\prime=\xb+\yb$ and ${\bf y}\in \ZZ^d$.

\medskip

Let $\fkp_i= T(\eb_i)$, where $\eb_i \in \ZZ^m$ denotes the $i$-th unit vector, and let us consider the prime divisor $\calD_i := \Spec(R/\fkp_i)$ on $\Spec R$. 
Then we see that the divisorial ideal $T(\ab)$ with $\ab=(a_1,\ldots,a_m)$ corresponds to the Weil divisor $-(a_1\calD_1+\cdots +a_m\calD_m)$. 
Moreover, by using the exact sequence \eqref{cl_seq}, we see that 
\begin{equation}\label{relation_Rdiv}
\begin{split}
\ZZ^d &\cong \left\{(b_1,\ldots,b_m) \in\ZZ^m : \sum_{i=1}^mb_i\calD_i=0 \text{ in } \Cl(R) \right\}, \text{ and } \\
\RR^d &\cong \left\{(b_1,\ldots,b_m) \in\RR^m : \sum_{i=1}^mb_i\calD_i=0 \text{ in } \Cl(R) \otimes_\ZZ \RR \right\}. 
\end{split}
\end{equation}
Remark that $\sum_{i=1}^mb_i\calD_i=0$ holds if and only if $(b_1,\ldots,b_m)=\tau(\xb)$ for some $\xb\in \ZZ^d$. 
%Thus, we may consider $\ub \in \RR^d$ as $(u_i)_i\in\RR^n$ such that $\sum_iu_i\calD_i=0$ in $\Cl(R)_\RR$ with $(u_i)_i=\tau(\ub)$. Hence, 
%\begin{equation*}
%\TT(\sigma(\ub))=\{(b_i)_i\in\ZZ^n \mid \sum_ib_i\calD_i=0 \ \text{in}\ \Cl(R),\ b_i\ge u_i \}. 
%\end{equation*}
%Here, we write $u_i=a_i+\delta_i$ with $a_i\in\ZZ$ and $\delta_i\in(-1,0]$, and set $c_i:= b_i-a_i\ge\delta_i$, $\alpha=-\sum_i a_i\calD_i\in\Cl(R)$. 
%Then, we see that 
%\begin{equation*}
%\TT(\sigma(\ub))=(a_i)_i+\{(c_i)_i\in\ZZ^n_{\ge0} \mid \sum_ic_i\calD_i=\alpha \}, 
%\end{equation*}
%and hence $T(\sigma(\ub))$ is isomorphic to the divisorial ideal corresponding to $\alpha=-\sum_i a_i\calD_i$, which is $T(a_1,\cdots,a_n)$. 
%Here, we note that $\ulcorner\sigma_i(\ub)\urcorner=a_i$.  
%Furthermore, we remark that $0=\sum_iu_i\calD_i=-\alpha+\sum_i\delta_i\calD_i$ in $\Cl(R)_\RR$, thus $\alpha=\sum_i\delta_i\calD_i$ in $\Cl(R)_\RR$. 
%We easily follow the converse of this argument. 

By using those descriptions, we can characterize what kinds of elements in $\Cl(R)$ correspond to conic divisorial ideals as follows. (See \cite[Subsection 2.1]{HN}.) 
\begin{lem}[{See \cite[Corollary 1.2]{Bru} and \cite[Proposition~3.2.3]{SmVdB}}]\label{lem:conic}
There exists a one-to-one correspondence among the following objects: 
\begin{enumerate}[\rm (1)]
\setlength{\parskip}{0pt} 
\setlength{\itemsep}{3pt}
\item a conic divisorial ideal $T(a_1,\ldots,a_m)$; 
\item an $\RR$-divisor $\sum_{i=1}^m\delta_i\calD_i$ with $(\delta_1,\ldots,\delta_m) \in(-1,0]^m$ up to equivalence, 
where we say that two $\RR$-divisors are equivalent if their difference is in \eqref{relation_Rdiv}; 
\item a full-dimensional cell of the decomposition of the semi-open cube $(-1,0]^d$ by 
hyperplanes $H_{i,q}=\{\xb \in \RR^d : \tau_i(\xb)=q \}$ for some $q \in\ZZ$ and $i=1,\ldots,m$.  
\end{enumerate}
\end{lem}
We identify the cell $\bigcap_{i=1}^m L_{i,a_i}$ with $T(a_1,\ldots,a_m)$, where $$L_{i,a_i}=\{\xb \in \RR^d : a_i-1<\tau_i(\xb) \leq a_i\}.$$ 

%As we mentioned in Theorem~\ref{motivation_thm}, conic divisorial ideals are precisely modules appearing in $R^{1/m}$ 
%as direct summands for $m\gg 0$. 
%Since $R^{1/m}$ is an MCM $R$-module, a conic divisorial ideal is also an MCM $R$-module. 
%We notice that the number of non-isomorphic conic divisorial ideals is finite because that of rank one MCM $R$-modules is finite \cite[Corollary~5.2]{BG1}. 
%Also, there exists a divisorial ideal that is a rank one MCM module but not conic (see e.g., \cite{Bae,Bru}).  
%If $\tau$ is simplicial, every divisorial ideal is conic, because a torsion element in $\operatorname{Cl}(R)$ is conic \cite[Theorem~3.2]{BG1}. 

\medskip

\subsection{Hyperplanes of the cone associated to $P_{K_{r_1,\ldots,r_n}}$}
In what follows, we consider $G=K_{r_1,\ldots,r_n}$ with $1 \leq r_1 \leq \cdots \leq r_n$, and assume that $n=3$ with $r_1 \geq 2$ or $n \geq 4$. 
Note that $G$ is non-bipartite. Let $V(G)=\bigsqcup_{i=1}^nV_i$ with $|V_i|=r_i$, let $V_i=\{v_{i1},\ldots,v_{ir_i}\}$ 
and let $E(G)=\{\{a,b\} : a \in V_i, b \in V_j \text{ for }1 \leq i \neq j \leq n\}$. 

In the sequel, we identify the entry of $\RR^d$ with the vertex of $G$ and assume that $v_{n,r_n}$ corresponds to the last ($d$-th) coordinate of $\RR^d$. 
\iffalse
we index the entries of $\RR^d$ in the following order: 
\begin{align*}
&\underbrace{v_{11},v_{12},\ldots,v_{1r_1}}_{V_1},\underbrace{v_{21},v_{22},\ldots,v_{2r_2}}_{V_2},\ldots,\underbrace{v_{n1},v_{n2},\ldots,v_{nr_n}}_{V_n} \\
&\longrightarrow \underbrace{x_1,\ldots,x_{r_1}}_{V_1},\underbrace{x_{r_1+1},\ldots,x_{r_1+r_2}}_{V_2},\ldots,\underbrace{x_{\sum_{i=1}^{n-1}r_i+1},\ldots,x_d}_{V_n}. 
\end{align*}
\fi

Let $\pi : \RR^d \rightarrow \RR^{d-1}$ with $\pi(x_1,\ldots,x_d)=(x_1,\ldots,x_{d-1})$. %i.e., let $\pi$ be the forgetting map of the last coordinate. 
For the proof of Theorem~\ref{main2}, we replace $P_G$ by the projected polytope $\pi(P_G)$. 
Let $$C_G = \ZZ_{\geq 0}\{(\alpha,1) \in \ZZ^d : \alpha \in \pi(P_G) \cap \ZZ^{d-1}\}$$ and consider the monoid $C_G \cap \ZZ^d$. 

First, we observe how the supporting hyperplanes of $C_G$ look like. 
We see that the variable ``$x_d$'' in $P_G$ changes into ``$2x_d-\sum_{i=1}^{d-1}x_i$'' since $\sum_{i=1}^dx_i=2$ holds. 
Hence, by \eqref{eq:hyp_P_G}, the system of supporting hyperplanes of $C_G$ becomes as follows:
\begin{equation}\label{eq:hyp_C_G}
\begin{split}
&x_i \geq 0 \;\; (i=1,\ldots,d-1), \quad 2x_d - \sum_{i=1}^{d-1}x_i \geq 0, \\
&x_d - \sum_{j \in V_i}x_j \geq 0 \;\; (i=1,\ldots,n-1), \quad \sum_{k \in V(G) \setminus V_n}x_k - x_d \geq 0. 
\end{split}
\end{equation}
Apply the following unimodular transformation: 
\begin{align*}
x_i \mapsto y_i \;\text{ for }i=1,\ldots,d-1, \;\text{ and }\;\; \sum_{k \in V(G) \setminus V_n}x_k - x_d \mapsto y_d. 
\end{align*}
Then \eqref{eq:hyp_C_G} changes as follows: 
\begin{equation}\label{eq:C_G}
\begin{split}
&y_i \geq 0 \;\; (i=1,\ldots,d), \\ 
&-y_d + \sum_{k \in V \setminus (V_i \sqcup V_n)}y_k \geq 0 \;\; (i=1,\ldots,n-1), \\
&-2y_d + \sum_{k \in V \setminus V_n}y_k - \sum_{u \in V_n \setminus \{v_{n,r_n}\}}y_u \geq 0. 
\end{split}
\end{equation}
Let $$C_G'=\{{\xb} \in \RR^d : \xb \text{ satisfies all inequalities in \eqref{eq:C_G}}\}.$$ 
Since the edge ring $\kk[G]$ is unimodularly equivalent to the monoid $\kk$-algebra $\kk[C_G' \cap \ZZ^d]$, we consider $C_G'$.

In what follows, let 
\begin{align}\label{eq:tau}
(\RR^d) \ni \tau_i=\begin{cases}
\eb_i &\text{ for }i=1,\ldots,d, \\
\sum_{k \in V \setminus V_n}\eb_k-\sum_{\ell \in V_{i-d}}\eb_\ell - \eb_d &\text{ for }i=d+1,\ldots,d+n. 
%\sum_{k \in V \setminus V_n}\eb_k - \sum_{\ell \in V_n}\eb_\ell - \eb_d &\text{ for }i=d+n. 
\end{cases}
\end{align}
Then each inequality in \eqref{eq:C_G} corresponds to $\langle \tau_i,\yb \rangle \geq 0$, where $\yb=(y_1,\ldots,y_d)$.

\bigskip

\subsection{Proof of Theorem~\ref{main2}}

Before proving Theorem~\ref{main2}, we describe $\calC(r_1,\ldots,r_n)$ more explicitly for small $n$'s. 
\begin{ex}\label{ex:conic}
Let $n=3$. Then 
\begin{equation}\label{ineq:n=3}
\begin{split}
\calC(r_1,r_2,r_3)=\{(z_1,z_2,z_3) \in \RR^3 : &-r_2 \leq z_1-z_2 \leq r_1, \; -r_3 \leq z_1-z_3 \leq r_1, \\ 
&-r_3 \leq z_2-z_3 \leq r_2, \; -r_2+1 \leq z_1 \leq 1,  \\
&-r_1+1 \leq z_2 \leq 1, \; -r_3+1 \leq z_1+z_2-z_3 \leq 1\}. 
\end{split}
\end{equation}
Note that the inequality $-r_2 + 1 \leq z_1 \leq 1$ (resp. $-r_1 + 1 \leq z_2 \leq 1$) 
comes from the second family in \eqref{conic_region} with $I=\{1\}$ (resp. $I=\{2\}$) and $J=\emptyset$ 
and the inequality $-r_3+1 \leq z_1+z_2-z_3 \leq 1$ comes from the third family in \eqref{conic_region} with $I=\{1,2\}$ and $J=\{3\}$. 

Let $n=4$. Then 
\begin{equation}\label{ineq:n=4}
\begin{split}
\calC(r_1,r_2,r_3,r_4)=\{&(z_1,z_2,z_3,z_4) \in \RR^4 : -r_j \leq z_i-z_j \leq r_i \text{ for }1 \leq i < j \leq 4, \\ 
&-\sum_{j \in \{1,2,3\} \setminus \{i\}}r_j+1 \leq z_i \leq 1 \text{ for }i=1,2,3, \\
&-2r_k \leq z_i+z_j-z_k \leq 2 \text{ for }\{i,j,k\}=\{1,2,3\}, \\
&-r_k-r_4+1 \leq z_i+z_j-z_4 \leq 1 \text{ for }\{i,j,k\}=\{1,2,3\}, \\
&-2r_4 \leq z_1+z_2+z_3-2z_4 \leq 2 \}. 
\end{split}
\end{equation}
Note that the third family of the inequalities $-2r_k \leq z_i+z_j-z_k \leq 2$ (as well as the fourth one) for $\{i,j,k\}=\{1,2,3\}$ 
is regarded as three inequalities. 
\end{ex}

Theorem~\ref{main2} directly follows from Lemma~\ref{key}, Proposition~\ref{prop:ab} and Lemma~\ref{n=34} below. 

\begin{lem}\label{key}
Let $\cb=(c_1,\ldots,c_n) \in \ZZ^n$. Assume that the following conditions (a) and (b) are equivalent: 
\begin{itemize}
\item[(a)] there exists $\xb \in (-1,0]^d$ such that $c_j-1 < \langle \tau_{d+j},\xb \rangle \leq c_j$ holds for $j=1,\ldots,n$; 
\item[(b)] $\cb \in \calC(r_1,\ldots,r_n) \cap \ZZ^n$. 
\end{itemize}
Then the conic divisorial ideals one-to-one correspond to the points in $\calC(r_1,\ldots,r_n) \cap \ZZ^n$. 
\end{lem}
\begin{proof}
Let $m=d+n$. 

\noindent
\underline{Conic $\Rightarrow$ $\calC(r_1,\ldots,r_n)$}: 
Take any $\ab=(a_1,\ldots,a_m) \in \Cl(R) \subset \ZZ^m$ (cf. \eqref{relation_Rdiv}) corresponding to a conic divisorial ideal $T(\ab)$. 
%We prove that $\ab \in \calC(r_1,\ldots,r_n) \cap \ZZ^n$ (cf. \eqref{conic_region}). 

We consider the decomposition of the semi-open cube $(-1,0]^d$ cut by the hyperplanes defined from $\tau_i \; (i=1,\ldots,m)$ in \eqref{eq:tau}. 
More precisely, by identifying a conic divisorial ideal $T(a_1,\ldots,a_m)$ with a full-dimensional cell of the decomposition $\bigcap_{i=1}^mL_{i,a_i}$, 
where $L_{i,a_i}=\{\xb \in \RR^d : a_i-1 < \langle \tau_i,\xb \rangle \leq a_i\}$ for $i=1,\ldots,m$, 
we analyze which $(a_1,\ldots,a_m) \in \ZZ^m$ defines a conic divisorial ideal. 

Here, we notice that $\bigcap_{i=1}^d L_{i,a_i} \subset (-1,0]^d$ holds if and only if $a_1=\cdots=a_d=0$, and in this case, we have $\bigcap_{i=1}^d L_{i,a_i} = (-1,0]^d$. 
Hence, we see that $a_1=\cdots=a_d=0$ and we may discuss the remaining linear forms $\tau_{d+1},\ldots,\tau_{d+n}$. 

In what follows, we show that $(a_{d+1},\ldots,a_{d+n}) \in \calC(r_1,\ldots,r_n) \cap \ZZ^n$. By definition of $\bigcap_{i=1}^mL_{i,a_i}$ and since it becomes full-dimensional, we see that 
\begin{align*}
%a_{d+i}&=\left\lceil -y_d + \sum_{k \in V \setminus (V_i \sqcup V_n)}y_k \right\rceil=\lceil \langle \tau_{d+i},\yb \rangle \rceil \text{ for }i=1,\ldots,n-1, \text{ and } \\
%a_{d+n}&=\left\lceil -2y_d + \sum_{k \in V \setminus V_n}x_k - \sum_{u \in V_n \setminus \{v_{n,r_n}\}}y_u \right\rceil=\lceil \langle \tau_{d+n},\yb \rangle \rceil, 
a_{d+i}-1<\langle \tau_{d+i},\yb \rangle \leq a_{d+i} \text{ for }i=1,\ldots,n, 
\end{align*}
where $\yb=(y_1,\ldots,y_d)$ and $-1 < y_i \leq 0$ for each $i=1,\ldots,d$. 
Therefore, we conclude that $(a_{d+1},\ldots,a_{d+n}) \in \calC(r_1,\ldots,r_n) \cap \ZZ^n$ since (a) implies (b). 

\medskip

\noindent
\underline{$\calC(r_1,\ldots,r_n)$ $\Rightarrow$ conic}: 
Take any $\ab \in \calC(r_1,\ldots,r_n) \cap \ZZ^n$. We show that a divisorial ideal $T(0,\ldots,0,a_1,\ldots,a_n)$ is conic. 
For this purpose, we prove that an $\RR$-divisor $\sum_{i=1}^mb_i\calD_i$ which is equivalent to $-\sum_{i=1}^na_i\calD_{d+i}$ 
satisfies that $-{\bf 1} < \bb \leq {\bf 0}$ (see Lemma~\ref{lem:conic}). 
By definition, we have $b_i+a_{i-d}=\langle \tau_i,\xb\rangle$ for each $i=1,\ldots,m$ for some $\xb \in \RR^d$, where we let $a_j=0$ if $j \leq 0$. 
Here, we have $b_i=\langle \tau_i,\xb\rangle=x_i$ for $i=1,\ldots,d$. 
Since we can choose $\xb \in \RR^d$ up to $\RR^d/\ZZ^d$, we may assume that $-1 < x_i \leq 0$ for $i=1,\ldots,d$. 
Hence, $-1 < b_i \leq 0$ holds for $i=1,\ldots,d$. 

By $\ab \in \calC(r_1,\ldots,r_n) \cap \ZZ^n$, since (b) implies (a), 
we see that $a_j -1 < \langle \tau_{d+j},\xb \rangle \leq a_j$ holds for each $j=1,\ldots,n$. Hence, 
\begin{align*}
-1< b_j=\langle \tau_{d+j},\xb\rangle - a_j \leq 0 \text{ for }j=1,\ldots,n, 
\end{align*}
as desired. 
\end{proof}
\begin{prop}\label{prop:ab}
The implication (a) $\Rightarrow$ (b) in Lemma~\ref{key} holds for any $n$. 
\end{prop}
\begin{proof}
%Assume that (a) holds. 
%
%For $i=1,\ldots,n-1$, since $0 \leq -x_j <1$ and $-1<x_j \leq 0$ for each $j \in V \setminus (V_i \sqcup V_n)$, we see from \eqref{eq:tau} that 
%$\displaystyle -\sum_{\substack{1 \leq j \leq n-1 \\ j \neq i}}r_j < \langle \tau_{d+i},\xb \rangle < 1$. 
%By this inequality, we obtain that $$-\sum_{\substack{1 \leq j \leq n-1 \\ j \neq i}}r_j + 1 \leq c_i \leq 1$$ holds for $i=1,\ldots,n-1$. 
%Similarly, since we have $-\sum_{k=1}^{n-1}r_k < \langle \tau_{d+n},\xb \rangle < r_n+1$, 
%we obtain that $$-\sum_{k=1}^{n-1}r_k+1 \leq c_n \leq r_n+1.$$ 
%
%Moreover, by our assumption, we have $\langle \tau_{d+i},\xb \rangle \leq c_i<\langle \tau_{d+i},\xb \rangle+1$ for each $i=1,\ldots,n$. Thus, we see that 
%\begin{align*}
%\langle \tau_{d+j}-\tau_{d+i},\xb \rangle - 1 < c_j-c_i < \langle \tau_{d+j}-\tau_{d+i},\xb \rangle + 1
%\end{align*}
%for each $1 \leq i, j \leq n$. Here, we observe that $\tau_{d+j}-\tau_{d+i}=\sum_{k \in V_i}\eb_k-\sum_{\ell \in V_j}\eb_\ell$ for each $1 \leq i, j \leq n$. Hence, we obtain that 
%\begin{align*}
%-r_i-1&< \sum_{k \in V_i} x_k - \sum_{\ell \in V_j} x_\ell -1 = \langle \tau_{d+j}-\tau_{d+i},\yb \rangle - 1 \\
%&< c_j-c_i < \langle \tau_{d+j}-\tau_{d+i},\xb \rangle + 1 = \sum_{k \in V_i} x_k - \sum_{\ell \in V_j} x_\ell +1 < r_j+1
%\end{align*}
%for each $1 \leq i, j \leq n$, as required. 
%
On the first family of the inequalities in \eqref{conic_region}, 
since we have $\langle \tau_{d+i},\xb \rangle \leq c_i<\langle \tau_{d+i},\xb \rangle+1$ for each $i=1,\ldots,n$ by our assumption, we see that 
\begin{align*}
\langle \tau_{d+j}-\tau_{d+i},\xb \rangle - 1 < c_j-c_i < \langle \tau_{d+j}-\tau_{d+i},\xb \rangle + 1
\end{align*}
for each $1 \leq i, j \leq n$. Here, we observe that 
\begin{align}\label{ineq:observe}\tau_{d+j}-\tau_{d+i}=\sum_{k \in V_i}\eb_k-\sum_{\ell \in V_j}\eb_\ell\text{ for each }1 \leq i, j \leq n.\end{align} 
Hence, we obtain that 
\begin{align*}
-r_i-1&< \sum_{k \in V_i} x_k - \sum_{\ell \in V_j} x_\ell -1 = \langle \tau_{d+j}-\tau_{d+i},\xb \rangle - 1 \\
&< c_j-c_i < \langle \tau_{d+j}-\tau_{d+i},\xb \rangle + 1 = \sum_{k \in V_i} x_k - \sum_{\ell \in V_j} x_\ell +1 < r_j+1
\end{align*}
for each $1 \leq i, j \leq n$. 

On the second family, for $I,J \subset [n-1]$ with $|I|=|J|+1$ and $I \cap J = \emptyset$, where $J$ is regarded as a multi-set, it follows from \eqref{ineq:observe} that 
\begin{align*}
\sum_{i \in I}\tau_{d+i}-\sum_{j \in J}\tau_{d+j}&=\tau_{d+i_0} + \sum_{\ell \in \bigcup_{j \in J}V_j}\eb_\ell - \sum_{k \in (\bigcup_{i \in I}V_i) \setminus V_{i_0}}\eb_k \\
&=\sum_{k \in V \setminus (V_n \cup \bigcup_{i \in I}V_i)}\eb_k+\sum_{\ell \in \bigcup_{j \in J}V_j}\eb_\ell - \eb_d, 
\end{align*}
where $i_0 \in I$ and $\bigcup_{j \in J}V_j$ is regarded as a multi-set. Similarly to the above discussions, we obtain that 
\begin{equation}\label{ineq:2}
\begin{split}
\sum_{i \in I}c_i - \sum_{j \in J}c_j &< \left\langle\sum_{i \in I}\tau_{d+i}-\sum_{j \in J}\tau_{d+j},\xb \right\rangle + |I| < 1 +|I|=|J|+2, \; \text{ and} \\
\sum_{i \in I}c_i - \sum_{j \in J}c_j &> \left\langle\sum_{i \in I}\tau_{d+i}-\sum_{j \in J}\tau_{d+j},\xb \right\rangle - |J| > - |J|- \sum_{i \in [n-1] \setminus I}r_i - \sum_{j \in J}r_j. 
\end{split}
\end{equation}

On the third family, for $I \subset [n-1]$ and $J \subset [n]$ with $|I|=|J|+1$, $n \in J$ and $I \cap J = \emptyset$, since we see from $n \in J$ that 
$$
\sum_{i \in I}\tau_{d+i}-\sum_{j \in J}\tau_{d+j}=\sum_{k \in V \setminus (V_n \cup \bigcup_{i \in I}V_i)}\eb_k+\sum_{\ell \in (\bigcup_{j \in J}V_j) \setminus \{v_{n,r_n}\}}\eb_\ell, 
$$
we obtain the conclusion by slightly modifying the estimation from \eqref{ineq:2}. 
\end{proof}
\begin{lem}\label{n=34}
The implication (b) $\Rightarrow$ (a) in Lemma~\ref{key} holds if $n=3$ or $n=4$. 
\end{lem}
\begin{proof}
Let $n=3$. Then $\calC(r_1,r_2,r_3)$ is explicitly described as in \eqref{ineq:n=3}. 
\iffalse
Then we can rewrite the inequalities \eqref{ineq:n=3} as follows: 
\begin{align*}
&-r_2+1 \leq z_1 \leq 1, \; -r_3+z_3 \leq z_1 \leq r_1+z_3, \\ 
&-r_1+1 \leq z_2 \leq 1, \; -r_3+z_3 \leq z_2 \leq r_2+z_3, \\
&-r_2 \leq z_1-z_2 \leq r_1, \; -r_3+z_3+1 \leq z_1+z_2 \leq z_3+1, \\
&-r_1-r_2+1 \leq z_3 \leq r_3+1
\end{align*}
\fi
%Note that the inequality $-r_1-r_2+1 \leq z_3 \leq r_3+1$ is redundant 
%since it follows from $-r_2+1 \leq z_1 \leq 1$ and $-r_1 \leq -z_1+z_3 \leq r_3$. 
By the direct computation, we can list the vertices of $\calC(r_1,r_2,r_3)$ as follows: 
\begin{align*}
&(1,1,1),  \; (1,1,r_3+1), \; (1,-r_1+1,-r_1+r_3+1), \; (-r_2+1,1,-r_2+r_3+1), \\
&(-r_2+1,-r_1+1,-r_1-r_2+1), \; (-r_2+1,-r_1+1,-r_1-r_2+r_3+1), \\ 
&(1,-r_1+1,-r_1+1), \; (-r_2+1,1,-r_2+1). 
\end{align*}
As mentioned in Remark~\ref{rem:vertex} below, it suffices to show the existence of $\xb$ satisfying (a) for those vertices. 
Given $\xb \in (-1,0]^d$, let 
\begin{align*}
y_1=\sum_{k \in V_1} x_k, \;\; y_2=\sum_{k \in V_2} x_k, \;\; y_3=\sum_{k \in V_3} x_k - x_d, \text{ and }y_d=x_d. 
\end{align*}
(Remark that $d=|V|=|V_1|+|V_2|+|V_3|=r_1+r_2+r_3$.) 
In our case, it suffices to show that for each vertex $(c_1,c_2,c_3) \in \calC(r_1,r_2,r_3)$, 
there is $\yb=(y_1,y_2,y_3,y_d) \in (-r_1,0] \times (-r_2,0] \times (-r_3+1,0]^d \times (-1,0]$ such that 
$$c_1-1<y_2-y_d \leq c_1, \;\; c_2-1 < y_1 - y_d \leq c_2 \;\text{ and }\;c_3-1 < y_1+y_2-y_3-2y_d \leq c_3.$$
We list how to choose such $\yb$'s for each vertex as follows: 
\begin{align*}
&(1,1,1): \yb=(0,0,0,-\epsilon), \;\;\; (1,1,r_3+1): \yb=(0,0,-r_3+1+\epsilon,-1+\epsilon), \\
&(1,-r_1+1,-r_1+r_3+1): \yb=(-r_1+\epsilon,0,-r_3+1+\epsilon,-1+\epsilon), \\ 
&(-r_2+1,1,-r_2+r_3+1): \yb=(0,-r_2+\epsilon,-r_3+1+\epsilon,-1+\epsilon), \\ 
&(-r_2+1,-r_1+1,-r_1-r_2+1): \yb=(-r_1+\epsilon,-r_2+\epsilon,0,-\epsilon), \\ 
&(-r_2+1,-r_1+1,-r_1-r_2+r_3+1) : \yb=(-r_1+\epsilon,-r_2+\epsilon,-r_3+1+\epsilon,-1+\epsilon),  \\ 
&(1,-r_1+1,-r_1+1) : \yb=(-r_1+\epsilon,0,0,-\epsilon), \\
&(-r_2+1,1,-r_2+1) : \yb=(0,-r_2+\epsilon,0,-\epsilon), 
\end{align*}
where $\epsilon>0$ is sufficiently small. 

For the case $n=4$, we may apply the same discussions as above by using \eqref{ineq:n=4}, 
although the computations become much more complicated. 
\end{proof}
\begin{rem}\label{rem:vertex}
For the proof of the implication (b) $\Rightarrow$ (a) in Lemma~\ref{key}, 
it is enough to show that (a) holds only for each vertex $\cb$ of $\calC(r_1,\ldots,r_n)$. 
In fact, once we can check (a) for all vertices, for any $\cb \in \calC(r_1,\ldots,r_n)$ written like 
$\cb=\sum_{v \in V(\calC)}r_vv$, where $V(\calC)$ denotes the set of vertices and $r_v \geq 0$ with $\sum_{v \in V(\calC)}r_v=1$, 
since there is $\xb_v \in (-1,0]^d$ with $c_j -1 < \langle \tau_{d+j},\xb_v \rangle \leq c_j$ ($j=1,\ldots,n$) for each $v \in V(\calC)$, 
we may set $\xb=\sum_{v \in V(\calC)}r_v\xb_v$. Then we can check $c_j -1 < \langle \tau_{d+j},\xb \rangle \leq c_j$ for each $j$. 
\end{rem}

\bigskip

%%%%%%%%%%%%%%%%%%%%%%%%%%%%%%%%%%%%%%%%%%%%%%%%%%%%%%%%%%%%%%%%%%%%%%%%%%%%%%%%%%%%%%%%%%%%%%%%%%%%%%%%%%%%%%%%%%%%%%%%%%%%%
%%%%%%%%%%%%%%%%%%%%%%%%%%%%%%%%%%%%%%%%%%%%%%%%%%%%%%%%%%%%%%%%%%%%%%%%%%%%%%%%%%%%%%%%%%%%%%%%%%%%%%%%%%%%%%%%%%%%%%%%%%%%%
%%%%%%%%%%%%%%%%%%%%%%%%%%%%%%%%%%%%%%%%%%%%%%%%%%%%%%%%%%%%%%%%%%%%%%%%%%%%%%%%%%%%%%%%%%%%%%%%%%%%%%%%%%%%%%%%%%%%%%%%%%%%%
%%%%%%%%%%%%%%%%%%%%%%%%%%%%%%%%%%%%%%%%%%%%%%%%%%%%%%%%%%%%%%%%%%%%%%%%%%%%%%%%%%%%%%%%%%%%%%%%%%%%%%%%%%%%%%%%%%%%%%%%%%%%%
%%%%%%%%%%%%%%%%%%%%%%%%%%%%%%%%%%%%%%%%%%%%%%%%%%%%%%%%%%%%%%%%%%%%%%%%%%%%%%%%%%%%%%%%%%%%%%%%%%%%%%%%%%%%%%%%%%%%%%%%%%%%%
%%%%%%%%%%%%%%%%%%%%%%%%%%%%%%%%%%%%%%%%%%%%%%%%%%%%%%%%%%%%%%%%%%%%%%%%%%%%%%%%%%%%%%%%%%%%%%%%%%%%%%%%%%%%%%%%%%%%%%%%%%%%%

\section{NCCR of Gorenstein edge rings of complete multipartite graphs}\label{sec:NCCR}

Finally, in this section, we give an NCCR for $\kk[K_{2,2,2}]$ and $\kk[K_4]$, i.e., we prove Theorem~\ref{main3}. 
Before it, we recall the definition of NCCRs. 
\begin{defi}
Let $A$ be a CM normal domain, let $M \neq 0$ be a reflexive $A$-module, and let $E=\End_A(M)$. 
\begin{itemize}
\item We call $E$ a {\em non-commutative resolution} ({\em NCR}, for short) of $A$ if the global dimension of $E$ is finite. 
\item Assume that $A$ is Gorenstein. 
Then $E$ is said to be a {\em non-commutative crepant resolution} ({\em NCCR}, for short) of $A$ 
if $E$ is an NCR and $E$ is a maximal CM $A$-module. 
\end{itemize}
\end{defi}

It scarcely happens that $E=\End_A(M_\calC)$ is an NCCR even if we know that $E$ is an NCR, 
where $M_\calC$ is the direct sum of all conic divisorial ideals. 
In principle, we have to remove some direct summands from $M_\calC$ to make $E$ an NCCR of $A$ (see, e.g., \cite{HN} and \cite{N}). 
In what follows, we will provide how to choose direct summands of $M_\calC$. 

\subsection{Preliminaries on non-commutative resolutions}

For our propose, we recall the methods developed in \cite{SpVdB}. 
We restrict the objects to the edge rings of complete multipartite graphs, 
although the theory in \cite{SpVdB} can be applied for more general objects.

Let $R=\kk[K_{r_1,\ldots,r_n}]$, where we assume that $n = 3$ with $r_1 \geq 2$ or $n \geq 4$. Namely, we have $\Cl(R) \cong \ZZ^n$ by Theorem~\ref{main1}. 
As before, let $d=\sum_{i=1}^nr_i$ and let $m=d+n$. 
Let $G=\Hom(\Cl(R),\kk^\times) \cong (\kk^\times)^n$, where $\kk^\times=\kk \setminus \{0\}$. 
Let $X(G)$ denote the character group of $G$. Then we know that $\Cl(R) \cong X(G)$. 
By using the surjection from $\ZZ^m$ to $\Cl(R) \cong X(G)$ (see \eqref{cl_seq}), we can send each prime divisor $\calD_i$ to $X(G)$, 
and we denote its image by $\beta_i$ for each $i=1,\ldots,m$. 
Let $V_\chi$ be the corresponding irreducible representation of a character $\chi \in X(G)$. 
We define the action of $G$ on $S=\kk[x_1,\ldots,x_m]$ by $g \cdot x_i = \beta_i(g)x_i$ for $g \in G$. Then we see that $R=S^G$. 

As discussed in \cite[Section 10]{SpVdB}, we can write conic divisorial ideals by using those as follows. 
Let $\calA=\mod(G,S)$ be the category of finitely generated $G$-equivariant $S$-modules. Given $\chi \in X(G)$, let 
\begin{align*}
P_\chi=V_\chi \otimes_\kk S \;\text{ and }\; M_\chi = P_\chi^G. 
\end{align*}
%let $M_\chi$ be an $R$-module of the form $(S \otimes_\kk V_\chi)^G$. 
Then $P_\chi \in \calA$. Note that $\chi$ corresponds to a Weil divisor $\sum_{i=1}^ma_i\calD_i$ and we have $M_{-\chi} \cong T(a_1,\ldots,a_m)$. 
For a subset $\calL \subset X(G)$, let 
$$\displaystyle P_\calL=\bigoplus_{\chi \in \calL}P_\chi \;\text{ and }\; \Lambda_\calL=\End_\calA(P_\calL).$$
Moreover, for $\chi \in X(G)$, let $P_{\calL,\chi}=\Hom_\calA(P_\calL,P_\chi)$. Furthermore, let 
$$M_\calL=\bigoplus_{\chi \in \calL}M_\chi \;\text{ and }\; E=\End(M_\calL).$$
Our goal is to choose $\calL$ such as $E$ becomes an NCCR. For this propose, we use the following lemmas:
\begin{lem}[{\cite[Lemma 10.1]{SpVdB}}]\label{lem:gldim}
If $\pdim_{\Lambda_\calL}P_{\calL,\chi}$ is finite for all $\chi \in X(G)$, then $\gldim \Lambda_\calL$ is also finite. 
\end{lem}

Let $Y(G)$ denote the group of one-parameter subgroups of $G$. Note that $Y(G) \cong \ZZ^n$. Let $Y(G)_\RR=Y(G) \otimes_\ZZ \RR \cong \RR^n$. 
We say that $\chi \in X(G)$ is {\em separated from $\calL$ by $\lambda \in Y(G)_\RR$} 
if it holds that $\langle \lambda,\chi \rangle < \langle \lambda,\chi' \rangle$ for each $\chi' \in \calL$.

To prove the finiteness of $\pdim_{\Lambda_\calL}P_{\calL,\chi}$ for all $\chi \in X(G)$, we use the following: 
\begin{lem}[{\cite[Lemma 10.2]{SpVdB}}]\label{lem:finite1}
Let $\chi \in X(G)$ be separated from $\calL$ by $\lambda \in Y(G)_\RR$. Then we obtain the acyclic complex 
\begin{align*}
0 \rightarrow \bigoplus_{\mu_1}P_{\calL,\mu_1} \rightarrow \cdots \rightarrow \bigoplus_{\mu_d}P_{\calL,\mu_d} \rightarrow P_{\calL,\chi} \rightarrow 0, 
\end{align*}
where for each $p \in [d_\lambda]$ with $d_\lambda=|\{ i \in [m] : \langle \beta_i,\lambda \rangle >0\}|$, we let 
$\mu_p=\chi+\beta_{i_1}+\cdots+\beta_{i_p}$ with $\{i_1,\ldots,i_p\} \subset [m]$. 
\end{lem}
By using this lemma, we see the following: 
\begin{lem}\label{lem:finite2}
If $\pdim_{\Lambda_\calL}P_\calL$ is finite for each $\chi \in \widetilde{\calL}$, where $\calL \subset \widetilde{\calL} \subset X(G)$, 
then $\pdim_{\Lambda_\calL}P_{\calL,\chi}$ is finite for all $\chi \in X(G)$. 
\end{lem}
Thus, one of our goals is to show the finiteness of $\pdim_{\Lambda_\calL}P_\calL$ 
for each $\chi \in \widetilde{\calL}$ to verify $\gldim E<\infty$ for a given $\widetilde{\calL}$.

Another goal is to show that $E$ is an MCM $R$-module. For this, we use the following: 
\begin{lem}\label{lem:MCM}
Assume that $\calL$ satisfies the following two conditions: 
\begin{itemize}
\item $\chi-\chi' \in \widetilde{\calL}$ for any $\chi,\chi' \in \calL$; 
\item $M_\chi$ is an MCM $R$-module for any $\chi \in \widetilde{\calL}$. 
\end{itemize}
Then $E$ is an MCM $R$-module. 
\end{lem}

\subsection{Proof of Theorem~\ref{main3}} 

This subsection is devoted to giving a proof of Theorem~\ref{main3}. 

First, let us describe the set of conic divisorial ideals in the cases $K_{2,2,2}$ and $K_4$. 
The direct computations imply that those correspond to the following set of the lattice points by Theorem~\ref{main2}: 
\begin{align*}
\calC(2,2,2) \cap \ZZ^3&=\{\pm(1,1,a) : a=1,2,3\} \cup \{\pm(1,0,a), \pm (0,1,a) : a=0,1,2\} \\
&\cup \{\pm(1,-1,a),(0,0,a) : a=-1,0,1\}, \\
\calC(1,1,1,1) \cap \ZZ^4&=\{\pm (1,1,1,2)\} \cup \{\pm (\alpha,1) : \alpha \in \{0,1\}^3\} \\
&\cup\{\pm(1,0,0,0),\pm(0,1,0,0),\pm(0,0,1,0),(0,0,0,0)\}. 
\end{align*}

As in Theorem~\ref{main3}, let 
\begin{align*}
\calL=\{(0,0,-1), (0,0,0), (1,0,0), (1,-1,0),(1,0,1),(0,-1,-1),(1,-1,-1),(0,-1,-2)\}
\end{align*} if $G=K_{2,2,2}$, and 
\begin{align*}
\calL=\{(0,0,0,0), (1,0,0,0), (1,0,0,1), (1,1,0,1), (1,1,1,2)\}
\end{align*}
if $G=K_4$. We set 
\begin{align*}
\widetilde{\calL}&=\calC(2,2,2) \cap \ZZ^3\text{ if $G=K_{2,2,2}$, and }\\
\widetilde{\calL}&=\calC(1,1,1,1) \cap \ZZ^3 \cup \{\pm(0,1,1,2)\}\text{ if $G=K_4$.}
\end{align*}
One can check by {\tt Macaulay2} (\cite{M2}) that the divisorial ideals corresponding to $\pm (0,1,1,2)$ are MCM $\kk[K_4]$-modules. 
Moreover, one can also verify that $\chi-\chi' \in \widetilde{\calL}$ holds for $\chi,\chi' \in \calL$ in both cases. 
Note that all conic divisorial ideals are rank one MCMs. Hence, $E$ is an MCM by Lemma~\ref{lem:MCM}. 

Our remaining task is to show that $\gldim E<\infty$. By Lemmas~\ref{lem:gldim}, \ref{lem:finite1} and \ref{lem:finite2}, we can conclude this if the following procedures terminate: 
\begin{enumerate}
%\item Set $\calL'=\calL$. 
\item Choose $\chi \in \widetilde{\calL} \setminus \calL$. (Note that $\pdim_{\Lambda_\calL}P_{\calL,\chi}=0$ if $\chi \in \calL$.) 
\item Find $\lambda \in Y(G)_\RR$ such that $\chi$ is separated from $\calL$ by $\lambda$ and 
$\chi + \beta_{i_1} + \cdots + \beta_{i_p} \in \calL$ for any $\{i_1,\ldots,i_p\} \subset [m]$ with $\langle \beta_{i_j},\lambda \rangle > 0$ for each $j$. 
\item If $\calL \cup \{\chi\} = \widetilde{\calL}$, then terminate the procedure. 
Otherwise, replace $\calL$ by $\calL \cup \{\chi\}$ and go back to (1). 
\end{enumerate}

For a while, we consider the case $K_{2,2,2}$. Then one has $m=d+n=6+3=9$. 
For the computations of $\beta_1,\ldots,\beta_9$, since we have \begin{align*}
\sum_{j=1}^9 \langle \tau_j,\eb_1 \rangle \calD_j &= \calD_1+\calD_8+\calD_9=0, \;\;
\sum_{j=1}^9 \langle \tau_j,\eb_2 \rangle \calD_j = \calD_2+\calD_8+\calD_9=0, \\
\sum_{j=1}^9 \langle \tau_j,\eb_3 \rangle \calD_j &= \calD_3+\calD_7+\calD_9=0, \;\;
\sum_{j=1}^9 \langle \tau_j,\eb_4 \rangle \calD_j = \calD_4+\calD_7+\calD_9=0, \\
\sum_{j=1}^9 \langle \tau_j,\eb_5 \rangle \calD_j &= \calD_5-\calD_9=0, \text{ and }
\sum_{j=1}^9 \langle \tau_j,\eb_6 \rangle \calD_j = \calD_6-\calD_7-\calD_8-2\calD_9=0, 
\end{align*}
where $\eb_1,\ldots,\eb_6 \in \RR^6$ denote the unit vectors of $\RR^6$ and $\tau_i$'s are as in \eqref{eq:tau}, 
by letting $\calD_7=(1,0,0)$, $\calD_8=(0,1,0)$ and $\calD_9=(0,0,1)$, we obtain that 
\begin{align*}
\calD_1=\calD_2=(0,-1,-1), \; \calD_3=\calD_4=(-1,0,-1), \; \calD_5=(0,0,1) \text{ and }\calD_6=(1,1,2). 
\end{align*}
Hence, we can choose $\beta_1,\ldots,\beta_9$ as the following multi-set: 
$$\{\beta_1,\ldots,\beta_9\}=\{(1,0,0),(0,1,0),(0,0,1) \times 2, (0,-1,-1) \times 2, (-1,0,-1) \times 2, (1,1,2)\},$$
where $\times 2$ stands for the duplicate. 

We list the ordering of choices of $\chi \in \widetilde{\calL} \setminus \calL$ and the corresponding $\lambda$ as follows: 

\medskip

\begin{minipage}{0.33\hsize}
\begin{center}
\begin{tabular}{lrr} \toprule
$\chi$ & $\lambda$ \\ \midrule
(0,-1,0) & (1,1,-1) \\
(1,1,1) & (0,-1,0) \\ 
(-1,-1,-2) & (1,0,0) \\ 
(0,1,1) & (0,-1,0) \\ 
(0,0,1) & (1,1,-1) \\ 
(0,1,2) & (0,-1,0) \\ 
\bottomrule
\end{tabular}
\end{center}
\end{minipage}
\begin{minipage}{0.33\hsize}
\begin{center}
\begin{tabular}{lrr} \toprule
$\chi$ & $\lambda$ \\ \midrule
(-1,-1,-1) & (1,0,0) \\ 
(-1,0,-1) & (1,0,0) \\ 
(-1,1,0) & (0,-1,0) \\ 
(0,1,0) & (0,-1,0) \\
(-1,0,0) & (1,1,-1) \\ 
(1,1,2) & (0,0,-1) \\
\bottomrule
\end{tabular}
\end{center}
\end{minipage}
\begin{minipage}{0.33\hsize}
\begin{center}
\begin{tabular}{lrr} \toprule
$\chi$ & $\lambda$ \\ \midrule
(-1,1,1) & (0,-1,0) \\
(-1,0,-2) & (0,-1,1) \\
(-1,-1,-3) & (0,0,1) \\
(1,0,2) & (0,1,-1) \\
(1,-1,1) & (0,1,-1) \\
(-1,1,-1) & (0,-1,0) \\
(1,1,3) & (0,0,-1) \\
\bottomrule
\end{tabular}
\end{center}
\end{minipage}

\medskip

\noindent
We read off the lines from top to bottom of the left-most table at first and go to the right. 
By this ordering, we can directly check that the procedure terminates. 

For example, at first, for $(0,-1,0) \in \widetilde{\calL} \setminus \calL$, we let $\lambda=(1,1,-1)$, 
and we see that $-1=\langle \lambda,(0,-1,0)\rangle < \langle \lambda,\chi'\rangle$ for each $\chi' \in \calL$ 
since $\langle \lambda,\chi'\rangle \in \{0,1\}$ for each $\chi'$. 
Moreover, we also have $\langle \beta,\lambda \rangle >0$ for $\beta \in \{\beta_1,\ldots,\beta_9\}$ if and only if $\beta \in \{(1,0,0),(0,1,0)\}$ 
and we can check that all of $\chi+(1,0,0)$, $\chi+(0,1,0)$ and $\chi+(1,1,0)$ belong to $\calL$. 
Thus, we add $(0,-1,0)$ to $\calL$. 
Next, take $\chi=(1,1,1) \in \widetilde{\calL} \setminus \calL$ and let $\lambda=(0,-1,0)$. 
Then we see that $-1=\langle \lambda,\chi \rangle < \langle \lambda,\chi' \rangle$ for each $\chi' \in \calL$ since $\langle \lambda,\chi' \rangle \in \{0,1\}$ for each $\chi'$. 
Moreover, we also have $\langle \beta,\lambda \rangle>0$ if and only if $\beta \in \{(0,-1,-1) \times 2\}$ 
and $\chi+(0,-1,-1)$ and $\chi+(0,-2,-2)$ belong to $\calL$. Thus, we add $(1,1,1)$ to $\calL$. 
We repeat this procedure until $\calL$ concides with $\widetilde{\calL}$. 

\medskip

In the case $K_4$, one has $m=4+4=8$. Since the method for the proof is completely the same as the case of $K_{2,2,2}$, 
we just list $\{\beta_1,\ldots,\beta_8\}$, the ordering of choices of $\chi \in \widetilde{\calL} \setminus \calL$ and the corresponding $\lambda$ below: 
\begin{align*}\{\beta_1,\ldots,\beta_8\}=\{&(1,0,0,0),(0,1,0,0),(0,0,1,0),(0,0,0,1), \\
&(1,1,1,2),(0,-1,-1,-1),(-1,0,-1,-1),(-1,-1,0,-1)\};\end{align*} 
\begin{minipage}{0.33\hsize}
\begin{center}
\begin{tabular}{lrr} \toprule
$\chi$ & $\lambda$ \\ \midrule
(1,1,1,1) & (0,-1,-1,1) \\
(0,1,1,1) & (2,0,0,-1) \\
(0,1,0,1) & (1,0,1,-1) \\ 
(0,0,0,1) & (1,1,0,-1) \\  
(0,0,0,-1) & (0,0,0,1) \\
(0,0,-1,-1) & (0,0,1,0) \\
(0,-1,-1,-1) & (0,1,0,0) \\
\bottomrule
\end{tabular}
\end{center}
\end{minipage}
\begin{minipage}{0.33\hsize}
\begin{center}
\begin{tabular}{lrr} \toprule
$\chi$ & $\lambda$ \\ \midrule
(-1,-1,-1,-1) & (1,0,0,0) \\ 
(-1,0,0,-1) & (1,0,0,0) \\ 
(-1,-1,-1,-2) & (0,0,0,1) \\
(-1,0,0,0) & (1,-1,0,0) \\
(-1,0,-1,-1) & (0,0,1,0) \\
(0,-1,0,0) & (0,2,0,-1) \\
(0,1,0,0) & (0,-2,0,1) \\ 
\bottomrule
\end{tabular}
\end{center}
\end{minipage}
\begin{minipage}{0.33\hsize}
\begin{center}
\begin{tabular}{lrr} \toprule
$\chi$ & $\lambda$ \\ \midrule
(1,0,1,1) & (-1,1,-1,0) \\
(0,0,1,1) & (1,1,0,-1) \\
(0,-1,0,-1) & (0,1,0,0) \\
(-1,-1,0,-1) & (1,1,0,-1) \\
(0,0,-1,0) & (0,0,2,-1) \\
(0,0,1,0) & (0,0,-2,1) \\
(0,1,1,2) & (1,-1,0,0) \\
(0,-1,-1,-2) & (0,0,0,1) \\
\bottomrule
\end{tabular}
\end{center}
\end{minipage}


\bigskip





\begin{thebibliography}{99}
\bibitem{Aus} M. Auslander, ``Representation Dimension of Artin Algebras'', Lecture Notes. Queen Mary College, London (1971). 
\bibitem{Bro} N. Broomhead, Dimer Model and Calabi--Yau Algebras, vol 215(1011). Mem. Amer. Math. Soc., Providence (2012). 
\bibitem{Bru} W. Bruns, Conic divisor classes over a normal monoid algebra, {\em Commutative algebra and algebraic geometry, Contemp. Math.}, {\bf 390}, Amer. Math. Soc., (2005), 63--71. 
\bibitem{BG1} W. Bruns and J. Gubeladze, Divisorial linear algebra of normal semigroup rings, {\em Algebra and Represent. Theory} {\bf 6} (2003), 139--168. 
\bibitem{BG2} W. Bruns and J. Gubeladze, Polytopes, rings and K-theory, Springer Monographs in Mathematics. Springer, Dordrecht, (2009).  
\bibitem{DITW} H. Dao, O. Iyama, R. Takahashi and M. Wemyss, Gorenstein modifications and $\QQ$-Gorenstein rings, {\em J. Algebraic Geom.} {\bf 29} (2020), 729--751. 
\bibitem{DH} E. De Negri and T. Hibi, Gorenstein algebras of Veronese type, {\em J. Algebra} {\bf 193} (1997), 629--639. 
\bibitem{FMS} E. Faber, G. Muller and K. E. Smith, Non-commutative resolutions of toric varieties, {\em Adv. Math.} {\bf 351} (2019), 236--274. 
\bibitem{M2} D. Grayson and M. Stillman. Macaulay2, a software system for research in algebraic geometry, Available at {\tt http://www.math.uiuc.edu/Macaulay2/}.
\bibitem{HHN} M. Hashimoto, T. Hibi and A. Noma, Divisor class groups of affine semigroup rings associated with distributive lattices, {\em J. Algebra} {\bf 149} (2), (1992), 352--357. 
\bibitem{HHO} J. Herzog, T. Hibi and H. Ohsugi, Binomial ideals, Graduate Texts in Mathematics, {\bf 279}. Springer, Cham, (2018). 
\bibitem{H87} T. Hibi, Distributive lattices, affine semigroup rings and algebras with straightening laws. In: Nagata, M., Matsumura, H. (eds.) Commutative Algebra and Combinatorics. Advanced Studies in Pure Mathematics, vol. 11, pp. 93--109. North-Holland, Amsterdam (1987). 
\bibitem{HL16} T. Hibi and N. Li, Unimodular Equivalence of Order and Chain polytopes, {\em Math. Scand.} {\bf 118}, No. 1 (2016), 5--12. 
\bibitem{HN} A. Higashitani and Y. Nakajima, Conic divisorial ideals of Hibi rings and their applications to non-commutative crepant resolutions, {\em Selecta Math.} {\bf 25} (2019), 25pp. 
\bibitem{IU} A. Ishii and K. Ueda, Dimer models and the special McKay correspondence, {\em Geom.Topol.} {\bf 19}, (2015), 3405--3466. 
\bibitem{IW} O. Iyama and M. Wemyss, Maximal modifications and Auslander--Reiten duality for non-isolated singularities, {\em Invent. Math.} {\bf 197}(3), (2014), 521--586. 
\bibitem{N} Y. Nakajima, Non-commutative crepant resolutions of Hibi rings with small class groups, {\em J. Pure Appl. Algebra} {\bf 223} (2019), 3461--3484. 
\bibitem{OH98} H. Ohsugi and T. Hibi, Normal polytopes arising from finite graphs, {\em J. Algebra} {\bf 207} (1998), 409--426. 
\bibitem{OH00} H. Ohsugi and T. Hibi, Compressed polytopes, initial ideals and complete multipartite graphs, {\em Illinois J. Math.} {\bf 44}, No. 2 (2000), 391--406. 
\bibitem{SVV} A. Simis, W. V. Vasconcelos and R. H. Villarreal, The integral closure of subrings associated to graphs, {\em J. Algebra} {\bf 199} (1998), 281--289. 
\bibitem{SmVdB} K. E. Smith and M. Van den Bergh, Simplicity of rings of differential operators in prime characteristic, {\em Proc. London Math. Soc.} (3) {\bf 75}, No. 1 (1997), 32--62.
\bibitem{SpVdB} \v{S}. \v{S}penko and M. Van den Bergh, Non-commutative resolutions of quotient singularities for reductive groups, {\em Invent. Math.} {\bf 210} (2017), no. 1, 3--67.  
\bibitem{SpVdB1} \v{S}. \v{S}penko and M. Van den Bergh, Non-commutative crepant resolutions for some toric singularities I, {\em Int. Math. Res. Not. IMRN}, to appear. 
\bibitem{SpVdB2} \v{S}. \v{S}penko and M. Van den Bergh, Non-commutative crepant resolutions for some toric singularities II, {\em J. Noncommut. Geom.} {\bf 14} (2020), no. 1, 73--103. 
\bibitem{S86} R. P. Stanley, Two Poset Polytopes, {\em Discrete Comput. Geom.} {\bf 1} (1986), 9--23. 
\bibitem{VdB} M. Van den Bergh, Non-Commutative Crepant Resolutions, The Legacy of Niels Henrik Abel, pp. 749--770. Springer, Berlin (2004). 
\bibitem{Villa} R. H. Villarreal, ``Monomial algebras'', Monographs and Research Notes in Mathematics. CRC Press, Boca Raton, FL, 2015.
\end{thebibliography}
\end{document}